\documentclass[11pt]{article}

\usepackage[margin=1.05in]{geometry}
\usepackage{amsmath,amssymb,amsthm,mathtools}
\usepackage{microtype}
\usepackage{enumitem}
\usepackage[hidelinks]{hyperref}
\usepackage{mathrsfs}

\newtheorem{theorem}{Theorem}[section]
\newtheorem{proposition}[theorem]{Proposition}
\newtheorem{lemma}[theorem]{Lemma}
\newtheorem{corollary}[theorem]{Corollary}

\newcommand{\R}{\mathbb R}
\newcommand{\C}{\mathbb C}
\newcommand{\Z}{\mathbb Z}
\newcommand{\1}{\mathbf 1}
\newcommand{\ip}[2]{\left\langle #1,#2\right\rangle}
\newcommand{\diag}{\operatorname{diag}}
\newcommand{\Dom}{\operatorname{Dom}}

\newcommand{\adj}[1]{\operatorname{adj}(#1)}

\title{Fuglede's conjecture holds for three intervals}
\author{Ruxi Shi}
\date{}

\begin{document}
\maketitle

\begin{abstract}
We prove that every bounded measurable subset of the real line that is both
spectral and a union of three intervals tiles the line by translations,
thereby completing Fuglede's conjecture for this class.
The proof develops a new cofactor-rigidity method based on the secular
determinant of an associated self-adjoint derivative.
\end{abstract}

\noindent\textbf{Keywords.}
Fuglede's conjecture, spectral sets, translational tilings, finite unions of
intervals, exponential polynomials, self-adjoint extensions.

\medskip
\noindent\textbf{2020 Mathematics Subject Classification.}
Primary 42A65, 52C22; Secondary 47B25, 47E05.

\section{Introduction}\label{sec:introduction}

A bounded measurable set $\Omega\subset\R$ of positive measure is called
\emph{spectral} if there is a set $\Lambda\subset\R$ such that, on putting
\[
   e_\lambda(x)=|\Omega|^{-1/2}e^{2\pi i\lambda x},
   \qquad
   E(\Lambda)=\{e_\lambda:\lambda\in\Lambda\},
\]
the family $E(\Lambda)$ is an orthonormal basis of $L^2(\Omega)$; then
$\Lambda$ is called a
\emph{spectrum} of $\Omega$.  The set $\Omega$ \emph{tiles by translations} if
there is a countable set $T\subset\R$ such that
\[
   \sum_{t\in T}\1_\Omega(x-t)=1
   \qquad\text{for almost every }x\in\R.
\]

Fuglede's spectral set conjecture asks whether a measurable set of finite
positive measure in $\R^d$ is spectral exactly when it tiles $\R^d$ by
translations.  Fuglede proved the equivalence when the spectrum or the
translation set is a full-rank lattice \cite{fuglede}.  Although the conjecture
was first disproved in the spectral-to-tiling direction in dimensions at least
five by Tao \cite{tao-counterexample}, and then in dimension four by Matolcsi
\cite{matolcsi-counterexample}, both directions are now known to fail for
general sets in every dimension $d\ge3$: there are spectral non-tiles
\cite{kolountzakis-matolcsi} and non-spectral tiles
\cite{farkas-matolcsi-mora} already in $\R^3$.  Its one-dimensional form
nevertheless remains subtle even for finite unions of intervals.

For convex bodies, the conjecture is known in every dimension.  Earlier work
established central-symmetry restrictions and treated planar domains and
three-dimensional polytopes
\cite{kolountzakis-convex,iosevich-katz-tao,greenfeld-lev}.  Lev and Matolcsi
ultimately proved the full result for all convex bodies
\cite{lev-matolcsi-convex}.

Finite abelian groups form an important discrete model for the conjecture and
have also supplied constructions that lift to Euclidean counterexamples.  Over
prime fields, the conjecture holds in $\Z_p^2$
\cite{iosevich-mayeli-pakianathan}, whereas counterexamples occur in
higher-dimensional vector spaces \cite{aten-et-al}.  For cyclic groups, notable
positive results include $\Z_{pqr}$ \cite{shi-pqr} and, more recently,
$\Z_{pqrs}$ \cite{kiss-et-al-pqrs}.  These finite-group results illustrate both
the arithmetic content of the conjecture and its close connection with the
one-dimensional problem.

For a union of two intervals, \L aba proved the full equivalence between
spectrality and translational tiling \cite{laba-two}.  Bose, Anil Kumar,
Krishnan, and Madan established the implication from tiling to spectrality for
three intervals and proved the reverse implication in several cases, including
three equal lengths, as well as under an additional hypothesis on the spectrum
\cite{bose-et-al}.  Bose and Madan subsequently treated all but one exceptional
configuration and related the remaining case to intersections of generalized
Vandermonde varieties on the three-torus \cite{bose-madan}.  Thus the general
implication from spectrality to tiling for three intervals was left open.
An important structural result in this direction is that spectra of finite
unions of intervals are periodic \cite{bose-madan-periodic}; in fact,
periodicity holds for every bounded spectral set in one dimension
\cite{iosevich-kolountzakis}.

There is also a substantial arithmetic approach to finite unions of
intervals.  \L aba related the one-dimensional spectral set conjecture to
multiplicative properties of roots of associated polynomials
\cite{laba-polynomials}.  For unions of finitely many unit intervals,
Konyagin and \L aba developed this connection further, reducing spectral and
tiling questions to cyclotomic-polynomial conditions and tilings of the
integers by finite sets \cite{konyagin-laba}.  Dutkay and Jorgensen showed that
the spectral-to-tiling implication in dimension one is equivalent to a
universal tiling conjecture and to corresponding assertions for restricted
classes such as finite unions of intervals with rational or integer endpoints
\cite{dutkay-jorgensen-universal}.

Our main result completes the proof of Fuglede's conjecture for unions of at
most three intervals.

\begin{theorem}\label{thm:main}
Let $\Omega\subset\R$ be a bounded measurable set which, up to a null set,
is a union of at most three bounded intervals.  If $\Omega$ is spectral, then $\Omega$
tiles $\R$ by translations.
\end{theorem}

Together with the tiling-to-spectral implication for three intervals proved in
\cite{bose-et-al}, Theorem~\ref{thm:main} establishes Fuglede's conjecture for
unions of three intervals.

The proof yields the following more precise geometric description.  As
usual, all tiling and fundamental-domain statements are understood up to
null sets.

\begin{theorem}\label{thm:geometric-trichotomy}
Let
\[
   \Omega=[a_1,b_1)\cup[a_2,b_2)\cup[a_3,b_3)
\]
be spectral, where the displayed intervals are its three pairwise disjoint
components.  Let $\ell_j=b_j-a_j>0$, and put
$L=\ell_1+\ell_2+\ell_3$.  After relabeling the three interval components if
necessary, at least one of the following geometric alternatives holds:
\begin{enumerate}[label=\textup{(\roman*)}]
\item $\Omega$ is a fundamental domain for the lattice $L\Z$;
\item $\ell_1=\ell_2=\ell_3=h=L/3$.  Equivalently, each component is a
fundamental domain for $h\Z$, and $\Omega$ is a three-fold lattice
multitile by $h\Z$;
\item $\ell_3=\ell_1+\ell_2=h=L/2$ and
$b_1\equiv a_2\pmod{h\Z}$.  On the circle $\R/h\Z$, the images of
$[a_1,b_1)$ and $[a_2,b_2)$ are adjacent half-open arcs that partition the
circle, while $[a_3,b_3)$ is a fundamental domain for $h\Z$.  Consequently,
$\Omega$ is a two-fold lattice multitile by $h\Z$.
\end{enumerate}
\end{theorem}

\paragraph{Contributions and relation to previous work.}
The description of spectral sets through self-adjoint realizations of the
derivative and unitary boundary conditions is classical; see
\cite{fuglede,dutkay-jorgensen,dutkay-jorgensen-momentum,
chakraborty-dutkay-jorgensen,chakraborty-dutkay-disconnected}.  Recent work has further developed the
connection with groups of local translations for finite unions of intervals
\cite{ducasse-dutkay-fernandez}.  The new contribution here is a rigidity
mechanism built on this framework.  The secular determinant encodes the
entire spectrum and yields both its density and a frequency-width uniqueness
principle.  Rank-one adjugates then give cofactor identities on the spectrum,
which uniqueness promotes to global identities and ultimately to a complete
trichotomy for the boundary unitary and the three interval lengths.  This
route differs from the earlier reduction of the remaining case to generalized
Vandermonde varieties \cite{bose-madan}: it derives the monomial,
equal-length, and additive-resonance alternatives directly from unrestricted
three-interval spectrality.  For a recent overview of the Fuglede problem and
its open directions, see \cite{kolountzakis-survey}.

\paragraph{Idea of the proof.}
The main logical chain is
\[
 \Lambda\longrightarrow A\longrightarrow U\longrightarrow F
 \longrightarrow \text{global cofactor identities}
 \longrightarrow \text{trichotomy}\longrightarrow \text{tiling}.
\]
Starting from a spectrum $\Lambda$, we diagonalize a self-adjoint realization
of $D=(2\pi i)^{-1}d/dx$ in the exponential basis.  Its domain is encoded by a
boundary unitary $U$ coupling the left and right endpoints of the interval
components.  This viewpoint belongs to the general theory of self-adjoint
extensions on unions of intervals; see, for example,
\cite{dutkay-jorgensen}.  The corresponding secular determinant is an
exponential polynomial whose zeros are exactly the points of $\Lambda$, all
simple.  A sharp zero count therefore gives the density of $\Lambda$ and a
uniqueness principle for exponential polynomials of smaller frequency width.

At each spectral point, the rank-one structure of the adjugate of the secular
matrix gives identities among its diagonal cofactors.  The uniqueness
principle makes these identities hold for every real parameter, converting the
analytic spectral information into algebraic restrictions on $U$.  Those
restrictions leave three possibilities: a monomial boundary unitary, three
equal interval lengths, or one additive relation among the lengths with a
specific block form for $U$.  The monomial case yields a lattice tiling by
endpoint periodization.  In the other two cases the spectrum is a union of two
or three lattice cosets and $\Omega$ is a lattice multitile; a fiber argument
then produces an ordinary translational tiling.  This route uses no prior
periodicity, rationality, or lattice-containment assumption on $\Lambda$.

\paragraph{Organization of the paper.}
Section~\ref{sec:notation} fixes the
interval notation and the harmless measure-theoretic reductions.
Section~\ref{sec:zero-count} establishes the zero-counting input, and
Section~\ref{sec:operator} constructs the spectral
derivative and its boundary unitary.  Sections~\ref{sec:cofactor} and
\ref{sec:classification} derive and classify the resulting algebraic
constraints.  Section~\ref{sec:fiber} proves the fiber lemma used to pass from
low-multiplicity lattice multitilings to ordinary tilings.
Section~\ref{sec:fewer} treats two components as a warm-up, and
Section~\ref{sec:completion} then completes the three-component case.
Section~\ref{sec:n-intervals} records which parts of the argument extend to
an arbitrary finite union of intervals and where the proof uses
three-component rigidity.

\section{Notation and reductions}\label{sec:notation}

Throughout, inner products are linear in the first variable.
For a square matrix $B$, let $B^{(j,k)}$ be the matrix obtained by deleting
row $j$ and column $k$.  The corresponding minor is $\det B^{(j,k)}$, and
the $(j,k)$ cofactor is
\[
 \operatorname{cof}_{jk}(B)=(-1)^{j+k}\det B^{(j,k)}.
\]
Thus the cofactor is the minor with its checkerboard sign.  We write $B^*$
for the Hermitian adjoint of $B$, whereas $\adj{B}$ denotes the adjugate,
namely the transpose of the cofactor matrix.  Equivalently,
\[
 (\adj{B})_{kj}=\operatorname{cof}_{jk}(B).
\]
The adjugate satisfies
\[
 B\,\adj{B}=\adj{B}\,B=(\det B)I.
\]

Overlapping or adjacent intervals may be merged, so it is enough to consider the decomposition of $\Omega$ into its distinct positive-length interval components.  The main part of the proof treats three components:
\[
   \Omega=\bigcup_{j=1}^3 [a_j,b_j),
   \qquad
   \ell_j=b_j-a_j>0,
   \qquad
   L=|\Omega|=\ell_1+\ell_2+\ell_3.
\]
Replacing $\Omega$ by a set differing from it by a null set does not change
$L^2(\Omega)$ or spectrality.  All translation sets constructed below are
countable.  Hence an almost-everywhere tiling obtained for one representative
also holds for every null-equivalent representative, because a countable
union of translates of the symmetric difference is null.  We may therefore
work throughout with disjoint half-open interval representatives.

\section{Zeros of exponential polynomials}\label{sec:zero-count}

The purpose of this section is not only to count zeros.  The resulting
estimate will give the exact density of the spectrum through the secular
determinant and, more importantly, a frequency-width uniqueness principle
that promotes identities valid on the spectrum to global
exponential-polynomial identities.

For an exponential polynomial
\[
 P(z)=\sum_{r=1}^N c_r e^{2\pi i\alpha_r z},
 \qquad z\in\C,\quad c_r\in\C,\quad \alpha_r\in\R,
\]
first combine terms having the same value of $\alpha_r$ and delete any terms
whose resulting coefficient is zero.  The remaining numbers $\alpha_r$ are
called the \emph{frequencies} of $P$, and together they form its
\emph{frequency set}.  If $P$ is nonzero, its \emph{frequency width} is the
diameter of this set, that is, its largest frequency minus its smallest
frequency.

We shall also use the elementary fact that exponential functions with
distinct real frequencies are linearly independent.  More precisely, if
$\alpha_1,\ldots,\alpha_N\in\R$ are pairwise distinct and
\[
   \sum_{r=1}^N c_r e^{2\pi i\alpha_r z}=0
   \qquad(z\in\C),
\]
then $c_1=\cdots=c_N=0$.  Indeed, taking derivatives of orders
$k=0,1,\ldots,N-1$ and evaluating at $z=0$ gives
\[
   \sum_{r=1}^N c_r(2\pi i\alpha_r)^k=0
   \qquad(k=0,1,\ldots,N-1).
\]
The coefficient matrix is a Vandermonde matrix and is invertible because
the $\alpha_r$ are pairwise distinct.

\begin{lemma}\label{lem:chebyshev}
Let $\beta_0<\cdots<\beta_m$ be real.  A nonzero real exponential polynomial
\[
   q(y)=\sum_{j=0}^m d_j e^{\beta_j y},
   \qquad d_j\in\R,
\]
has at most $m$ real zeros, counted with multiplicity.
\end{lemma}

\begin{proof}
The assertion follows by induction on $m$, the case $m=0$ being immediate.
Multiplication by $e^{-\beta_0y}$ does not change the zeros or their
multiplicities.  Differentiate the resulting function.  If the derivative
vanishes identically, the function is a nonzero constant and has no zeros.
Otherwise, its derivative is a nonzero exponential polynomial with at most
$m$ terms and distinct real exponents, and hence has at most $m-1$ zeros by
induction.

For any finite collection of zeros of the normalized function having total
multiplicity $N$, the multiplicity version of Rolle's theorem
\cite[Chapter~I]{karlin-studden} gives at least
$N-1$ zeros for the derivative of that function: a zero of multiplicity $r$ contributes
$r-1$, and between two successive distinct zeros there is one further zero.
Thus $N-1\le m-1$.  Since this applies to every finite collection, the total
number of zeros is at most $m$.
\end{proof}

\begin{lemma}\label{lem:zero-count}
Let $m$ be a nonnegative integer, let
$\alpha_0,\ldots,\alpha_m\in\R$ satisfy
$\alpha_0<\alpha_1<\cdots<\alpha_m$, and let
$c_0,\ldots,c_m\in\C$ satisfy $c_0\ne0$ and $c_m\ne0$.  Define the entire function
$P:\C\to\C$ by
\[
   P(z)=\sum_{j=0}^m c_j e^{2\pi i\alpha_j z},
   \qquad z\in\C.
\]
There exists $Y>0$ such that every zero $z$ of $P$ satisfies
$|\operatorname{Im}z|<Y$.  For $R\ge0$, let $N_P(R)$ be the number of zeros
of $P$, counted with multiplicity, that satisfy
$|\operatorname{Re}z|\le R$.  Then, as $R\to\infty$,
\[
   N_P(R)=2(\alpha_m-\alpha_0)R+O(1).
\]
Here the constant implicit in $O(1)$ may depend on $P$ but is independent of
$R$.
In particular, if $Z_{\R}(P)$ is the set of distinct real zeros, then
\[
 \limsup_{R\to\infty}
 \frac{\#(Z_{\R}(P)\cap[-R,R])}{2R}
 \le \alpha_m-\alpha_0.
\]
\end{lemma}

\begin{proof}
If $m=0$, then $P(z)=c_0e^{2\pi i\alpha_0z}$ has no zeros, and all the
conclusions are immediate.  We therefore assume that $m\ge1$.

Choose $Y>0$ so large that
\[
 \sum_{j=1}^m\frac{|c_j|}{|c_0|}
 e^{-2\pi(\alpha_j-\alpha_0)Y}<1,
 \qquad
 \sum_{j=0}^{m-1}\frac{|c_j|}{|c_m|}
 e^{-2\pi(\alpha_m-\alpha_j)Y}<1.
\]
Such a choice is possible because every exponent occurring in these two
sums is negative.  Since
\[
 |c_j e^{2\pi i\alpha_j(x+iy)}|=|c_j|e^{-2\pi\alpha_jy}
\]
independently of $x$, the first inequality says that the $\alpha_0$ term
strictly dominates the sum of the remaining terms when $y=Y$, and the
second gives the corresponding dominance of the $\alpha_m$ term when
$y=-Y$.  The relevant ratios decrease further when, respectively, $y\ge Y$
and $y\le-Y$.  Hence the same strict dominance holds throughout both
exterior half-planes.  In particular, $P$ has no zero there, and every zero
satisfies $|\operatorname{Im}z|<Y$.

We apply the argument principle to the positively oriented rectangle with
vertices $\pm R\pm iY$, first taking $R$ so that neither vertical side
contains a zero.  On its lower side, traversed from $-R-iY$ to $R-iY$, write
\[
 P(z)=c_m e^{2\pi i\alpha_mz}
 \left(1+\sum_{j=0}^{m-1}\frac{c_j}{c_m}
 e^{2\pi i(\alpha_j-\alpha_m)z}\right).
\]
By the choice of $Y$, the sum inside the parentheses has modulus strictly
less than $1$, uniformly in $\operatorname{Re}z$.  The parenthetical factor
therefore remains in a fixed disk centered at $1$ that does not contain the
origin, so its change of argument is bounded independently of $R$.  The
exponential factor contributes $2\pi\alpha_m(2R)=4\pi\alpha_mR$.  Thus
\[
   \Delta\arg P=4\pi\alpha_m R+O(1).
\]
On the upper side we instead factor out
$c_0e^{2\pi i\alpha_0z}$.  Because this side is traversed from $R+iY$ to
$-R+iY$, the same calculation gives
\[
   \Delta\arg P=-4\pi\alpha_0 R+O(1).
\]

It remains to bound the changes of argument on the vertical sides uniformly
in $R$.  On a zero-free vertical line $\operatorname{Re}z=x$, write
\[
   f_x(y)=P(x+iy)=\sum_{j=0}^m d_j e^{-2\pi\alpha_j y},
   \qquad d_j=c_je^{2\pi i\alpha_jx},
   \qquad -Y\le y\le Y.
\]
Choose $\theta$ so that
\[
   q_\theta(y)=\operatorname{Im}(e^{-i\theta}f_x(y))
\]
is not identically zero and does not vanish at $y=\pm Y$; only finitely many
directions are excluded.  This is a real exponential polynomial, so
Lemma~\ref{lem:chebyshev} gives at most $m$ zeros.  Since $f_x$ is
nonvanishing, choose a continuous argument $f_x(y)=|f_x(y)|e^{i\phi(y)}$.
The equality $q_\theta(y)=0$ is equivalent to
$\phi(y)\in\theta+\pi\mathbb Z$.  Every value $\theta+k\pi$ strictly between
$\phi(-Y)$ and $\phi(Y)$ is therefore attained, by the intermediate value
theorem, at a zero of $q_\theta$; different such values are attained at
different points.  There are at least
$|\phi(Y)-\phi(-Y)|/\pi-1$ such values.  Since $q_\theta$ has at most $m$
zeros, it follows that
\[
 |\phi(Y)-\phi(-Y)|\le(m+1)\pi.
\]
Thus the vertical changes of argument are bounded uniformly in $x$, and
hence in $R$.

The total change of argument around the rectangle is therefore
\[
   4\pi(\alpha_m-\alpha_0)R+O(1).
\]
Dividing this change by $2\pi$, the argument principle gives
\[
 N_P(R)=2(\alpha_m-\alpha_0)R+O(1)
\]
for these admissible values of $R$.  For an
arbitrary $R$, move each vertical side slightly to a zero-free line.  The
total multiplicity of the zeros on any fixed vertical line is at most $m$.
Indeed, write the restriction as $f_x(y)=P(x+iy)$.  If $y_0$ is a zero of
multiplicity $r$, then locally
\[
   f_x(y)=(y-y_0)^r g(y),
   \qquad g(y_0)\ne0.
\]
For all but one direction $\theta$ modulo $\pi$,
$\operatorname{Im}(e^{-i\theta}g(y_0))\ne0$, and hence
$\operatorname{Im}(e^{-i\theta}f_x(y))$ has a zero of exactly the same
multiplicity $r$ at $y_0$.  There are only finitely many zeros on the compact
vertical segment, so a single direction can be chosen simultaneously for all
of them.  Lemma~\ref{lem:chebyshev} then bounds their total multiplicity by
$m$.

Since the zeros of the nonzero entire function $P$ are discrete, each
vertical side can be displaced so slightly that no zeros other than those on
the original side are crossed.  The preceding bound shows that this changes
the count by at most $2m$, independently of $R$.  The formula follows for
every $R$, with an $O(1)$ constant depending on $P$ but not on $R$.

Finally, every distinct real zero in $[-R,R]$ is among the zeros counted by
$N_P(R)$ and contributes multiplicity at least one.  Therefore
\[
 \#(Z_{\R}(P)\cap[-R,R])\le N_P(R).
\]
Dividing by $2R$ and taking the upper limit gives the asserted density bound
for the distinct real zeros.
\end{proof}

\begin{corollary}\label{cor:density-unique}
Let $D>0$, and let $\Lambda\subset\R$ satisfy
\[
   \#(\Lambda\cap[-R,R])=2DR+O(1)
   \qquad(R\to\infty).
\]
Let
\[
   Q(\zeta)=\sum_{r=1}^N c_r e^{2\pi i\alpha_r\zeta},
   \qquad \zeta\in\C,\quad c_r\in\C,\quad \alpha_r\in\R,
\]
be an exponential polynomial satisfying
\[
   Q(\lambda)=0\qquad\text{for every }\lambda\in\Lambda.
\]
Then either $Q$ vanishes identically on $\C$, or its frequency width is at
least $D$.  In particular, if its frequency width is strictly smaller than
$D$, then $Q$ vanishes identically.
\end{corollary}

\begin{proof}
Suppose that $Q$ is not identically zero, and let its frequency width be
$W$.  If $W<D$, Lemma~\ref{lem:zero-count} gives at most
$2WR+O(1)$ distinct real zeros
in $[-R,R]$, whereas the assumed set $\Lambda$ contributes
$2DR+O(1)$ distinct zeros.  This is impossible for large $R$, so $W\ge D$.
\end{proof}

\section{The self-adjoint derivative and its boundary unitary}\label{sec:operator}

The correspondence between spectra and self-adjoint realizations of the
derivative goes back to Fuglede \cite{fuglede}.  Its boundary-unitary form for
unions of intervals was developed systematically by Dutkay and Jorgensen
\cite{dutkay-jorgensen}.  We recall the construction in the notation needed
here.  The new use of this framework begins with the secular determinant in
Subsection~\ref{subsec:secular}: its zero set yields the density and rigidity
statements that drive the later classification.

\subsection{The spectral and minimal derivative operators}

Recall that
\[
   \Omega=\bigcup_{j=1}^3[a_j,b_j),
   \qquad \ell_j=b_j-a_j,
   \qquad L=|\Omega|=\ell_1+\ell_2+\ell_3.
\]
Assume now that $\Omega$ has spectrum $\Lambda$ and retain the notation
$e_\lambda$ introduced above.
Before constructing the associated operator, we record a direct indication
of why a unitary coupling of the endpoint values should exist.  For
$\lambda\in\Lambda$, put
\[
   v_a(\lambda)
   =\bigl(e^{2\pi i\lambda a_1},e^{2\pi i\lambda a_2},
           e^{2\pi i\lambda a_3}\bigr)^{\mathsf T},
   \qquad
   v_b(\lambda)
   =\bigl(e^{2\pi i\lambda b_1},e^{2\pi i\lambda b_2},
           e^{2\pi i\lambda b_3}\bigr)^{\mathsf T}.
\]
If $\lambda\ne\mu$ are in $\Lambda$, their orthogonality gives
\[
  0=\int_\Omega e^{2\pi i(\lambda-\mu)x}\,dx
   =\frac{\ip{v_b(\lambda)}{v_b(\mu)}_{\C^3}
           -\ip{v_a(\lambda)}{v_a(\mu)}_{\C^3}}
          {2\pi i(\lambda-\mu)}.
\]
For $\lambda=\mu$, every coordinate of $v_a(\lambda)$ and
$v_b(\lambda)$ has modulus one, so
$\|v_a(\lambda)\|_{\C^3}^2=\|v_b(\lambda)\|_{\C^3}^2=3$.
Consequently, the Gram matrices of the two endpoint families,
\[
 G_a=\bigl(\ip{v_a(\lambda)}{v_a(\mu)}_{\C^3}\bigr)_
          {\lambda,\mu\in\Lambda},
 \qquad
 G_b=\bigl(\ip{v_b(\lambda)}{v_b(\mu)}_{\C^3}\bigr)_
          {\lambda,\mu\in\Lambda},
\]
agree entry by entry:
\[
 \ip{v_a(\lambda)}{v_a(\mu)}_{\C^3}
 =
 \ip{v_b(\lambda)}{v_b(\mu)}_{\C^3}
 \qquad(\lambda,\mu\in\Lambda).
\]
Therefore the assignment
$v_a(\lambda)\mapsto v_b(\lambda)$ defines an isometry between their linear
spans.  The self-adjoint extension argument below makes this endpoint
coupling a uniquely determined unitary on $\C^3$.

Define an operator $A$ by
\[
   Af=\sum_{\lambda\in\Lambda}
      \lambda\ip{f}{e_\lambda}e_\lambda,
\]
with domain
\[
   \Dom(A)=\left\{f\in L^2(\Omega):
      \sum_{\lambda\in\Lambda}\lambda^2
      |\ip{f}{e_\lambda}|^2<\infty\right\}.
\]
This is the maximal natural domain for the displayed spectral formula.  Indeed,
the expansion of $f$ in the given orthonormal basis is
\[
   f=\sum_{\lambda\in\Lambda}\ip{f}{e_\lambda}e_\lambda.
\]
The series for $Af$ is formally obtained by multiplying its $\lambda$th
Fourier coefficient by $\lambda$, and Parseval's identity gives
\[
   \|Af\|_2^2
   =\sum_{\lambda\in\Lambda}\lambda^2
      |\ip{f}{e_\lambda}|^2.
\]
Thus the defining condition is exactly the condition that the spectral series
for $Af$ represent an element of $L^2(\Omega)$.
Because $L^2(\Omega)$ is separable, the orthonormal basis indexed by
$\Lambda$ is countable.  The coefficient map
\[
 \mathcal F:L^2(\Omega)\longrightarrow\ell^2(\Lambda),
 \qquad
 \mathcal Ff=(\ip{f}{e_\lambda})_{\lambda\in\Lambda},
\]
is unitary, and $A=\mathcal F^{-1}M_\lambda\mathcal F$, where
$M_\lambda$ is the coordinate multiplication operator defined by
\[
 (M_\lambda c)_\nu=\nu c_\nu\qquad(\nu\in\Lambda),
\]
with domain
\[
 \Dom(M_\lambda)
 =\left\{c=(c_\nu)_{\nu\in\Lambda}\in\ell^2(\Lambda):
     \sum_{\nu\in\Lambda}\nu^2|c_\nu|^2<\infty\right\}.
\]
Thus $A$ is self-adjoint with exactly the domain displayed above.
Let
\[
   D=\frac{1}{2\pi i}\frac{d}{dx}.
\]
Define the initial derivative operator $S_0$ by
\[
   \Dom(S_0)=\bigoplus_{j=1}^3 C_c^\infty(a_j,b_j),
   \qquad S_0f=Df,
\]
and let $S=\overline{S_0}$, where the bar denotes the operator closure:
equivalently, $\Dom(S_0)$ is completed in the graph norm, rather than merely
in the $L^2$ norm.  The graph norm
\[
   \|f\|_D^2=\|f\|_2^2+\|Df\|_2^2
\]
is equivalent on each component to the $H^1$ norm.  Since
$C_c^\infty(a_j,b_j)$ is dense in $H_0^1(a_j,b_j)$, it follows that
\[
   \Dom(S)=\bigoplus_{j=1}^3 H_0^1(a_j,b_j),
   \qquad Sf=Df,
\]
where the derivative is understood weakly.  Thus $S$ is the minimal closed
derivative, and its domain consists precisely of the functions
$f=(f_1,f_2,f_3)$ such that $f_j\in H^1(a_j,b_j)$ and
$f_j(a_j)=f_j(b_j)=0$ for $j=1,2,3$.

For $f\in\Dom(S)$, integration by parts gives
\[
   \ip{Df}{e_\lambda}=\lambda\ip{f}{e_\lambda}.
\]
Parseval's identity therefore yields
\[
 \sum_{\lambda\in\Lambda}\lambda^2
       |\ip{f}{e_\lambda}|^2=\|Df\|_2^2<\infty,
\]
so $f\in\Dom(A)$ and $Af=Df$.  Hence $S\subset A$.  Since $A=A^*$,
taking adjoints gives $A\subset S^*$.  The adjoint operator $S^*$ is the
maximal derivative operator:
\[
   \Dom(S^*)=\bigoplus_{j=1}^3 H^1([a_j,b_j]),
   \qquad
   S^*f=Df
\]
on each interval, with no boundary conditions imposed.  Consequently, every element of
$\Dom(A)$ belongs to this maximal Sobolev domain and $A$ acts there as $D$.
For $f\in\Dom(S^*)$, define its left- and right-endpoint vectors by
\[
 f(\mathbf a)=
 \begin{pmatrix}f(a_1)\\f(a_2)\\f(a_3)\end{pmatrix},
 \qquad
 f(\mathbf b)=
 \begin{pmatrix}f(b_1)\\f(b_2)\\f(b_3)\end{pmatrix}.
\]
For $f,g\in\Dom(S^*)$, Green's formula is
\begin{equation}\label{eq:green}
\ip{Df}{g}-\ip{f}{Dg}
 =\frac{1}{2\pi i}
 \left(
   \ip{f(\mathbf b)}{g(\mathbf b)}_{\C^3}
   -\ip{f(\mathbf a)}{g(\mathbf a)}_{\C^3}
 \right).
\end{equation}

\subsection{Unitary boundary conditions}

The statement below is the three-interval instance of the general
unitary-boundary parametrization of self-adjoint extensions of the derivative
on a finite union of intervals; see
\cite[Section~3]{dutkay-jorgensen}.  The proof given here is a self-contained
argument adapted to the particular operator $A$ defined from the spectrum
$\Lambda$: it derives the boundary unitary directly from the trace map,
Green's formula, and the maximal isotropic boundary space.  Thus it differs
in presentation and starting point from the general extension-theoretic
treatment in \cite{dutkay-jorgensen}, while also fixing our convention
$f(\mathbf b)=Uf(\mathbf a)$.

\begin{lemma}\label{lem:boundary-unitary}
There is a unique unitary matrix $U\in U(3)$ such that
\begin{equation}\label{eq:domain-U}
\Dom(A)=\left\{f\in\bigoplus_{j=1}^3H^1([a_j,b_j]):
  f(\mathbf b)=Uf(\mathbf a)
\right\}.
\end{equation}
\end{lemma}

\begin{proof}
For $f\in\Dom(S^*)=\bigoplus_{j=1}^3H^1([a_j,b_j])$, define its
boundary trace by
\[
   \operatorname{Tr}f=\bigl(f(\mathbf a),f(\mathbf b)\bigr)
   \in\C^3\oplus\C^3.
\]
The trace map from $\Dom(S^*)$ onto $\C^3\oplus\C^3$ is surjective:
on each interval, an affine function can be chosen with any prescribed
pair of endpoint values.  Moreover,
\[
   \ker(\operatorname{Tr})
   =\Dom(S)
   =\bigoplus_{j=1}^3H_0^1(a_j,b_j);
\]
indeed, its kernel consists exactly of the functions whose values vanish at
all six endpoints.
Set
\[
   \mathcal L=\operatorname{Tr}(\Dom(A))
   \subset\C^3\oplus\C^3.
\]
This is a linear subspace.  If
$\operatorname{Tr}f=(x,y)$ and $\operatorname{Tr}g=(x',y')$ for
$f,g\in\Dom(A)$, then $Af=Df$, $Ag=Dg$, and the symmetry of $A$ gives
\[
   \ip{Df}{g}-\ip{f}{Dg}
   =\ip{Af}{g}-\ip{f}{Ag}=0.
\]
Green's formula \eqref{eq:green} therefore yields
\[
   \ip{y}{y'}_{\C^3}-\ip{x}{x'}_{\C^3}=0.
\]
Thus the Hermitian form
\[
   \omega\bigl((x,y),(x',y')\bigr)
   =
   \ip{y}{y'}_{\C^3}-\ip{x}{x'}_{\C^3}
\]
vanishes on $\mathcal L\times\mathcal L$; this is precisely the statement
that $\mathcal L$ is isotropic.  Define its orthogonal complement with
respect to $\omega$ by
\[
   \mathcal L^{\perp_\omega}
   =
   \left\{z\in\C^3\oplus\C^3:
   \omega(\ell,z)=0\ \text{for every }\ell\in\mathcal L\right\}.
\]
Since $\mathcal L$ is isotropic,
$\mathcal L\subseteq\mathcal L^{\perp_\omega}$.  To prove the reverse
inclusion, let $z\in\mathcal L^{\perp_\omega}$.  By surjectivity of the
trace map, choose $h\in\Dom(S^*)$ such that $\operatorname{Tr}h=z$.
Green's formula then gives
\[
 \ip{Af}{h}=\ip{f}{Dh}\qquad(f\in\Dom(A)),
\]
so the definition of the adjoint gives
$h\in\Dom(A^*)$ and $A^*h=Dh$.  Since $A=A^*$, we have
$h\in\Dom(A^*)=\Dom(A)$, and hence
\[
   \operatorname{Tr}h\in\operatorname{Tr}(\Dom(A))=\mathcal L.
\]
Thus $z\in\mathcal L$, proving
$\mathcal L^{\perp_\omega}\subseteq\mathcal L$.  Consequently,
\[
   \mathcal L=\mathcal L^{\perp_\omega},
\]
so $\mathcal L$ is maximal isotropic.  With respect to the decomposition
$\C^3\oplus\C^3$, the form $\omega$ is represented by
\[
   \begin{pmatrix}
      -I_3&0\\
      0&I_3
   \end{pmatrix}.
\]
This matrix has three negative eigenvalues and three positive eigenvalues,
all nonzero.  Thus $\omega$ has signature $(3,3)$ and is nondegenerate on
the six-dimensional space $\C^3\oplus\C^3$.  Hence
\[
   \dim\mathcal L+\dim\mathcal L^{\perp_\omega}=6.
\]
Since $\mathcal L=\mathcal L^{\perp_\omega}$, it follows that
$2\dim\mathcal L=6$, and therefore $\dim\mathcal L=3$.

Consider the projection onto the first factor,
\[
   \pi_a:\mathcal L\longrightarrow\C^3,
   \qquad \pi_a(x,y)=x.
\]
This projection is injective.  Indeed, if $(0,y)\in\mathcal L$, then
isotropy applied to this vector and itself gives
\[
   0=\omega\bigl((0,y),(0,y)\bigr)=\|y\|_{\C^3}^2,
\]
and hence $y=0$.  Since both $\mathcal L$ and $\C^3$ have dimension three,
the injective linear map $\pi_a$ is bijective.  Consequently, for every
$x\in\C^3$ there is a unique $y\in\C^3$ such that $(x,y)\in\mathcal L$.
Define $Ux=y$.  The linearity of $\mathcal L$ and the uniqueness of $y$
show that $U:\C^3\to\C^3$ is linear, and
\[
   \mathcal L=\{(x,Ux):x\in\C^3\}.
\]

For $x,x'\in\C^3$, the two vectors $(x,Ux)$ and $(x',Ux')$ belong to
$\mathcal L$.  Their isotropy gives
\[
   0=\omega\bigl((x,Ux),(x',Ux')\bigr)
    =\ip{Ux}{Ux'}_{\C^3}-\ip{x}{x'}_{\C^3}.
\]
Thus $U$ preserves inner products.  In particular, $U$ is an isometry and
is injective; since its domain and codomain both have dimension three, it is
surjective and therefore unitary.  The graph $\mathcal L$ also determines
$U$ uniquely, because for each $x$ it contains exactly one vector of the
form $(x,y)$.

It remains to verify the asserted equality of domains.  If
$f\in\Dom(A)$, then by the definition of $\mathcal L$,
\[
   \operatorname{Tr}f\in\mathcal L,
\]
and hence $f(\mathbf b)=Uf(\mathbf a)$.  Conversely, suppose that
$h\in\Dom(S^*)$ satisfies $h(\mathbf b)=Uh(\mathbf a)$.  Then
$\operatorname{Tr}h\in\mathcal L=\operatorname{Tr}(\Dom(A))$, so there
exists $g\in\Dom(A)$ with $\operatorname{Tr}g=\operatorname{Tr}h$.
Therefore
\[
   h-g\in\ker(\operatorname{Tr})=\Dom(S)\subset\Dom(A).
\]
Since also $g\in\Dom(A)$, we conclude that $h\in\Dom(A)$.  This proves the
domain equality in \eqref{eq:domain-U}.
\end{proof}

We call this uniquely determined matrix $U$ the \emph{boundary unitary}
associated with the spectral pair $(\Omega,\Lambda)$.

The two descriptions of $\Dom(A)$ express the same operator from complementary
viewpoints.  The first is its spectral description, determined by the given
basis $E(\Lambda)$, whereas \eqref{eq:domain-U} is its differential
description: functions have $H^1$ regularity on each component and their six
endpoint values satisfy three boundary conditions.  The boundary unitary $U$ records how
the values at the three left endpoints are coupled to those at the three right
endpoints.  Its unitarity is precisely what makes the boundary term in
\eqref{eq:green} vanish, while maximality of the boundary condition gives
self-adjointness rather than mere symmetry.  Thus $U$ may be viewed as a
lossless coupling of the interval endpoints.

For each $\lambda\in\Lambda$, one has
$e_\lambda\in\Dom(A)$ and $Ae_\lambda=\lambda e_\lambda$.  Applying
Lemma~\ref{lem:boundary-unitary} and canceling the common normalization
factor gives the boundary relation
\begin{equation}\label{eq:global-boundary}
  \begin{pmatrix}
  e^{2\pi i\lambda b_1}\\
  e^{2\pi i\lambda b_2}\\
  e^{2\pi i\lambda b_3}
  \end{pmatrix}
  =U
  \begin{pmatrix}
  e^{2\pi i\lambda a_1}\\
  e^{2\pi i\lambda a_2}\\
  e^{2\pi i\lambda a_3}
  \end{pmatrix},
  \qquad \lambda\in\Lambda.
\end{equation}

\subsection{The secular determinant}\label{subsec:secular}

For $\zeta\in\C$, set
\[
   z_j(\zeta)=e^{2\pi i\zeta\ell_j},
   \qquad
   D_\ell(\zeta)=\diag(z_1(\zeta),z_2(\zeta),z_3(\zeta)),
\]
and define the \emph{boundary secular matrix} and its determinant by
\begin{equation}\label{eq:F}
   M(\zeta)=D_\ell(\zeta)-U,
   \qquad
   F(\zeta)=\det M(\zeta).
\end{equation}
The parameter $\zeta$ is allowed to be complex so that $F$ is an entire
function; the spectral points will be identified below with its real zeros.
We refer to $F$ as the \emph{secular determinant}.  As shown below, the
equation $F(\zeta)=0$ is precisely the characteristic equation for the
spectrum of the associated self-adjoint derivative.

The following proposition is the main bridge between the spectral data and
the algebraic analysis of the boundary unitary.  It identifies the spectrum
with the zero set of a single exponential polynomial and supplies the
simplicity and density properties used throughout the rest of the proof.

\begin{proposition}\label{prop:secular}
The secular determinant $F$ has the following properties:
\begin{enumerate}[label=\textup{(\roman*)}]
\item For every $\lambda\in\R$,
\[
   \lambda\in\Lambda\quad\text{if and only if}\quad F(\lambda)=0.
\]
\item All zeros of $F$ are real and simple.
\item For every $\lambda\in\Lambda$, the matrix $M(\lambda)$ has rank two.
Moreover, if
\[
 y_\lambda=(e^{2\pi i\lambda a_1},e^{2\pi i\lambda a_2},
             e^{2\pi i\lambda a_3})^{\mathsf T},
\]
then there exists a nonzero scalar $\kappa(\lambda)$ such that
\[
 \adj{M(\lambda)}
 =\kappa(\lambda)y_\lambda(D_\ell(\lambda)y_\lambda)^*.
\]
\item The spectrum satisfies
\begin{equation}\label{eq:density}
   \#(\Lambda\cap[-R,R])=2LR+O(1).
\end{equation}
\end{enumerate}
\end{proposition}

\begin{proof}
\emph{\textup{(i)}, the reality assertion in \textup{(ii)}, and
\textup{(iii)}.}
For any $\mu\in\C$, a solution of $Df=\mu f$ in $\Dom(S^*)$ has the
form
\[
   f(x)=c_j e^{2\pi i\mu x}
   \quad\text{on }[a_j,b_j],
   \qquad c_j\in\C,\quad j=1,2,3.
\]
Indeed, on the $j$th interval the equation is
$\frac{d}{dx}f_j=2\pi i\mu f_j$ in the weak sense.  Hence the weak derivative of
$e^{-2\pi i\mu x}f_j(x)$ is zero, so this function is constant on
$[a_j,b_j]$; denote the constant by $c_j$.
Define
\[
   y=(y_1,y_2,y_3)^{\mathsf T},
   \qquad
   y_j=c_je^{2\pi i\mu a_j}.
\]
Then the left-endpoint vector is
\[
   f(\mathbf a)=y.
\]
Since $b_j=a_j+\ell_j$, the value at the right endpoint of the $j$th
interval is
\[
   f(b_j)
   =c_je^{2\pi i\mu b_j}
   =e^{2\pi i\mu\ell_j}c_je^{2\pi i\mu a_j}
   =z_j(\mu)y_j.
\]
Hence
\[
   f(\mathbf b)=D_\ell(\mu)y.
\]
The boundary condition $f(\mathbf b)=Uf(\mathbf a)$ from
Lemma~\ref{lem:boundary-unitary} is therefore equivalent to
\[
   D_\ell(\mu)y=Uy,
   \qquad\text{or equivalently}\qquad
   M(\mu)y=0.
\]
Conversely, if $0\ne y\in\ker(D_\ell(\mu)-U)$, define
$f_y:\Omega\to\C$ componentwise by
\[
 f_y(x)=y_j e^{2\pi i\mu(x-a_j)}
 \qquad
 (x\in[a_j,b_j],\quad j=1,2,3).
\]
Each component belongs to $H^1([a_j,b_j])$, so
$f_y\in\Dom(S^*)$.
For each $j=1,2,3$,
\[
   f_y(a_j)=y_j,
   \qquad
   f_y(b_j)=y_je^{2\pi i\mu\ell_j}=z_j(\mu)y_j.
\]
Therefore
\[
   f_y(\mathbf a)=y,
   \qquad
   f_y(\mathbf b)=D_\ell(\mu)y.
\]
Since $y\in\ker(D_\ell(\mu)-U)$, we have
$D_\ell(\mu)y=Uy$.  Thus
$f_y(\mathbf b)=Uf_y(\mathbf a)$, so
Lemma~\ref{lem:boundary-unitary} gives $f_y\in\Dom(A)$ and
$Af_y=\mu f_y$.  Hence the map $y\mapsto f_y$ is a linear isomorphism
\[
   \ker M(\mu)\longrightarrow\ker(A-\mu I).
\]
Its inverse is the left-endpoint map $f\mapsto f(\mathbf a)$.
Therefore,
\[
 \ker(D_\ell(\mu)-U)\ne\{0\}
 \quad\Longleftrightarrow\quad
 \ker(A-\mu I)\ne\{0\}.
\]
By definition, $\mu$ is an eigenvalue of $A$ precisely when there exists a
nonzero $f\in\Dom(A)$ such that $Af=\mu f$, or equivalently, when
$\ker(A-\mu I)\ne\{0\}$.
Also, since $M(\mu)$ is a square matrix,
$F(\mu)=\det M(\mu)=0$ precisely when $\ker M(\mu)\ne\{0\}$.
Consequently, for every $\mu\in\C$,
\[
\begin{aligned}
 F(\mu)=0
 &\quad\Longleftrightarrow\quad
 \ker(D_\ell(\mu)-U)\ne\{0\}\\
 &\quad\Longleftrightarrow\quad
 \mu\text{ is an eigenvalue of }A\\
 &\quad\Longleftrightarrow\quad
 \mu\in\Lambda.
\end{aligned}
\]
To verify the last equivalence explicitly, recall that
$A=\mathcal F^{-1}M_\lambda\mathcal F$, where
\[
  (M_\lambda c)_\nu=\nu c_\nu
  \qquad(c=(c_\nu)_{\nu\in\Lambda}\in\Dom(M_\lambda),
          \nu\in\Lambda).
\]
Thus $M_\lambda c=\mu c$ is equivalent to
$(\nu-\mu)c_\nu=0$ for every $\nu\in\Lambda$.  This equation has a nonzero
solution $c$ precisely when $\mu\in\Lambda$; in that case its solution space
is spanned by the coordinate vector supported at $\mu$.  Unitary equivalence
therefore shows that $\mu$ is an eigenvalue of $A$ precisely when
$\mu\in\Lambda$, and that every eigenspace of $A$ is one-dimensional.
The isomorphism
$\ker M(\lambda)\cong\ker(A-\lambda I)$ established above therefore gives
$\dim(\ker M(\lambda))=1$ for every $\lambda\in\Lambda$.  Since
$M(\lambda)$ is a $3\times3$ matrix, the rank--nullity theorem yields
$\operatorname{rank}M(\lambda)=2$, proving \textup{(iii)}.
Moreover, $A$ is self-adjoint, meaning that $A=A^*$.  If $F(\mu)=0$, then
$\mu$ is an eigenvalue of $A$, so there is a nonzero $f\in\Dom(A)$ with
$Af=\mu f$.  Hence
\[
 \mu\lVert f\rVert^2
 =\ip{Af}{f}
 =\ip{f}{Af}
 =\overline{\mu}\lVert f\rVert^2,
\]
where the middle equality uses $A=A^*$.  Since $f\ne0$, we obtain
$\mu=\overline{\mu}$, so every zero of $F$ is real.

\emph{The simplicity assertion in \textup{(ii)}.}
Fix $\lambda\in\Lambda$ and define
\[
   y=(e^{2\pi i\lambda a_1},e^{2\pi i\lambda a_2},e^{2\pi i\lambda a_3})^{\mathsf T}.
\]
Then $M(\lambda)y=0$.  Since $D_\ell(\lambda)y=Uy$ and
$D_\ell(\lambda)$ is unitary,
\[
   M(\lambda)^*D_\ell(\lambda)y
   =(D_\ell(\lambda)^*-U^*)D_\ell(\lambda)y
   =y-U^*Uy=0.
\]
Since $M(\lambda)$ has rank two, its right kernel is one-dimensional.
Because $M(\lambda)y=0$ and $y\ne0$,
\[
 \ker M(\lambda)=\operatorname{span}\{y\}.
\]
Similarly, the preceding calculation and $D_\ell(\lambda)y\ne0$ give
\[
 \ker M(\lambda)^*=\operatorname{span}\{D_\ell(\lambda)y\}.
\]
Also, $\det M(\lambda)=0$, and the adjugate identity from
Section~\ref{sec:notation} gives
\[
 M(\lambda)\adj{M(\lambda)}=0,
 \qquad
 \adj{M(\lambda)}M(\lambda)=0.
\]
The first equality shows that every column of $\adj{M(\lambda)}$ is a
multiple of $y$.  The second shows that every row is a multiple of
$(D_\ell(\lambda)y)^*$.  Hence $\adj{M(\lambda)}$ must have the form
\[
 \adj{M(\lambda)}=\kappa(\lambda)y(D_\ell(\lambda)y)^*.
\]
Finally, rank two means that some $2\times2$ minor of $M(\lambda)$ is
nonzero.  Since these minors are the entries of its adjugate up to signs,
$\adj{M(\lambda)}\ne0$, and therefore $\kappa(\lambda)\ne0$.
We have
\[
   M'(\lambda)=2\pi i\,\diag(\ell_1,\ell_2,\ell_3)D_\ell(\lambda),
\]
where the derivative of a matrix-valued function is understood entrywise:
if $B(\zeta)=(b_{jk}(\zeta))_{j,k}$, then
$B'(\zeta)=(b_{jk}'(\zeta))_{j,k}$.  Consequently,
\begin{equation}\label{eq:transversal}
 (D_\ell(\lambda)y)^*M'(\lambda)y
   =2\pi i\sum_{j=1}^3\ell_j
      |z_j(\lambda)y_j|^2
   =2\pi i(\ell_1+\ell_2+\ell_3)
   =2\pi iL\ne0.
\end{equation}
Here $\lambda\in\Lambda\subset\R$, so
$|z_j(\lambda)|=|y_j|=1$ for $j=1,2,3$.
Applying Jacobi's cofactor formula\footnote{For a differentiable square-matrix
function $B(\zeta)=(b_{jk}(\zeta))$, the formula is
$\displaystyle \frac{d}{d\zeta}\det B(\zeta)
=\sum_{j,k}\operatorname{cof}_{jk}(B(\zeta))b_{jk}'(\zeta)
=\operatorname{tr}\!\bigl(\adj{B(\zeta)}B'(\zeta)\bigr)$.
Because this cofactor form does not use $B(\zeta)^{-1}$, it remains valid
when $B(\zeta)$ is singular.} to $B=M$ at $\zeta=\lambda$ gives
\[
\begin{aligned}
 F'(\lambda)
 &=\operatorname{tr}(\adj{M(\lambda)}\,M'(\lambda))\\
 &=\kappa(\lambda)(D_\ell(\lambda)y)^*M'(\lambda)y\ne0.
\end{aligned}
\]
The second equality uses
$\operatorname{tr}(uv^*B)=v^*Bu$ for column vectors $u,v$.
Hence the zero is simple.

\emph{\textup{(iv)}.}
As an exponential polynomial, $F$ has lowest frequency $0$, with coefficient
$-\det U\ne0$, and highest frequency $L$, with coefficient $1$.  Indeed,
these terms arise by choosing, respectively, all entries from $-U$ and all
three diagonal exponentials; every other determinant term uses a proper,
nonempty subset of the lengths and therefore has frequency strictly between
$0$ and $L$.  Lemma~\ref{lem:zero-count}, together with the fact that all
zeros are real and simple and equal to $\Lambda$, now gives
\eqref{eq:density}.
\end{proof}

Proposition~\ref{prop:secular}\textup{(i)--(iii)} allows us henceforth to
replace questions about the spectrum by identities for the exponential
polynomial $F$, and its rank and simplicity conclusions drive the cofactor
identities in the next section.  The density conclusion in
Proposition~\ref{prop:secular}\textup{(iv)} makes the zero-counting uniqueness
principle applicable to those identities.

\section{Cofactor rigidity}\label{sec:cofactor}

This section contains the central rigidity step of the proof.  The rank-one
adjugate formula first gives cofactor relations at the spectral points.  We
then combine the density of the spectrum with frequency-width uniqueness to
show that these relations are identities on the whole complex plane.

Recall that $U\in U(3)$ is the boundary unitary associated with the spectral
pair $(\Omega,\Lambda)$ by Lemma~\ref{lem:boundary-unitary}.  Retain the
notation $M(\zeta)$ from \eqref{eq:F}, and write $U=(u_{jk})$.
Let $C_j(\zeta)$ denote the $(j,j)$ cofactor of $M(\zeta)$.
For a diagonal cofactor the checkerboard sign is positive; for example,
\[
   C_1(\zeta)
   =\det\begin{pmatrix}
      z_2(\zeta)-u_{22}&-u_{23}\\
      -u_{32}&z_3(\zeta)-u_{33}
   \end{pmatrix}
   =(z_2(\zeta)-u_{22})(z_3(\zeta)-u_{33})-u_{23}u_{32}.
\]
By the notation fixed in Section~\ref{sec:notation},
$C_j(\zeta)=(\adj{M(\zeta)})_{jj}$.
We shall also use the three principal $2\times2$ minors of $U$:
\[
   \Delta_1=u_{22}u_{33}-u_{23}u_{32},\quad
   \Delta_2=u_{11}u_{33}-u_{13}u_{31},\quad
   \Delta_3=u_{11}u_{22}-u_{12}u_{21}.
\]
Equivalently, $\Delta_j=\operatorname{cof}_{jj}(U)$ for
$j\in\{1,2,3\}$.

The next proposition is important because its conclusion holds for every
$\zeta\in\C$, not only for $\lambda\in\Lambda$.  It converts discrete
spectral constraints into identities of entire exponential polynomials,
which can then be used in the algebraic classification of
Section~\ref{sec:classification}.

\begin{proposition}\label{prop:cofactor}
For every $\zeta\in\C$,
\begin{equation}\label{eq:cofactor-identities}
   z_1(\zeta)C_1(\zeta)
   =z_2(\zeta)C_2(\zeta)
   =z_3(\zeta)C_3(\zeta).
\end{equation}
\end{proposition}

\begin{proof}
Fix $\lambda\in\Lambda$.  By Proposition~\ref{prop:secular}\textup{(iii)},
\[
 \adj{M(\lambda)}
 =\kappa(\lambda)y_\lambda(D_\ell(\lambda)y_\lambda)^*
\]
for the vector $y_\lambda$ defined there and some
$\kappa(\lambda)\ne0$.  For each $j\in\{1,2,3\}$,
$|(y_\lambda)_j|=1$, and the corresponding diagonal entry is
\begin{align*}
   C_j(\lambda)
   &=(\adj{M(\lambda)})_{jj}\\
   &=\kappa(\lambda)(y_\lambda)_j
       \overline{z_j(\lambda)(y_\lambda)_j}\\
   &=\kappa(\lambda)\overline{z_j(\lambda)}.
\end{align*}
Since $\lambda$ is real,
\[
   z_j(\lambda)\overline{z_j(\lambda)}=1,
   \qquad
   z_j(\lambda)C_j(\lambda)=\kappa(\lambda)
   \qquad(j=1,2,3).
\]
Thus \eqref{eq:cofactor-identities} holds at every point of $\Lambda$.

It remains to promote these equalities to identities in $\zeta\in\C$.
For brevity in the following expansions, write
$z_j=z_j(\zeta)$ and $C_j=C_j(\zeta)$.  Consider the three pairwise
differences
\[
 z_1C_1-z_2C_2,\qquad z_1C_1-z_3C_3,\qquad z_2C_2-z_3C_3.
\]
Expanding the $2\times2$ determinants that define the $C_j$ expresses each
difference as an exponential polynomial whose frequencies can be read off
explicitly.  The three expansions are
\begin{align}
 z_1C_1-z_2C_2
   &=-u_{22}z_1z_3+\Delta_1z_1
     +u_{11}z_2z_3-\Delta_2z_2,\label{eq:H12}\\
 z_1C_1-z_3C_3
   &=-u_{33}z_1z_2+\Delta_1z_1
     +u_{11}z_2z_3-\Delta_3z_3,\label{eq:H13}\\
 z_2C_2-z_3C_3
   &=-u_{33}z_1z_2+\Delta_2z_2
     +u_{22}z_1z_3-\Delta_3z_3.\label{eq:H23}
\end{align}
Indeed, since $z_j(\zeta)=e^{2\pi i\zeta\ell_j}$, the four terms on the
right-hand side of \eqref{eq:H12}, in the order displayed, have frequencies
\[
   \ell_1+\ell_3,\quad \ell_1,\quad
   \ell_2+\ell_3,\quad \ell_2.
\]
Some of these four numbers may coincide.  In that case the terms with the
same frequency are combined, and their combined coefficient may even vanish.
Consequently, if $z_1C_1-z_2C_2$ is not identically zero, then every
frequency $\alpha$ occurring with a nonzero coefficient in its reduced
expansion must satisfy
\[
 \alpha\in
 \{\ell_1+\ell_3,\ell_1,\ell_2+\ell_3,\ell_2\}.
\]
Writing $m_{12}=\min(\ell_1,\ell_2)$, all four numbers lie in
$[m_{12},L-m_{12}]$.  Therefore, if $z_1C_1-z_2C_2$ is not identically
zero, its frequency width is at most
\[
 L-2m_{12}<L.
\]
Similarly, if $z_1C_1-z_3C_3$ or $z_2C_2-z_3C_3$ is not identically zero,
its frequency width is at most $L-2\min(\ell_1,\ell_3)$ or
$L-2\min(\ell_2,\ell_3)$, respectively.  All three bounds are strictly
smaller than $L$.  Each of these three exponential polynomials vanishes on
$\Lambda$.  By Proposition~\ref{prop:secular}\textup{(iv)}, the density
parameter of $\Lambda$ in Corollary~\ref{cor:density-unique} is $D=L$.
Thus, if any one of the three polynomials were nonzero, that corollary would
give a contradiction because its frequency width is strictly smaller than
$D$.  Hence all three vanish identically, and
\eqref{eq:cofactor-identities} holds for every $\zeta\in\C$.
\end{proof}

Thus the infinitely many constraints indexed by the spectrum have been
compressed into global exponential-polynomial identities.  The remaining
problem is finite-dimensional and algebraic.

\section{Algebraic classification of the boundary unitary}\label{sec:classification}

By Proposition~\ref{prop:cofactor}, the pairwise differences on the
left-hand sides of \eqref{eq:H12}--\eqref{eq:H23} vanish identically.
Consequently, each right-hand side is an identity of the form
\[
   \sum_r c_r e^{2\pi i\alpha_r\zeta}=0
   \qquad(\zeta\in\C),
\]
where the real numbers $\alpha_r$ are the frequencies.  The following lemma
treats the case in which no two frequencies in any one of the three
identities coincide.  A unitary matrix is called \emph{monomial} if every
row and every column has exactly one nonzero entry.

\begin{lemma}\label{lem:monomial-from-minors}
Suppose $U\in U(3)$ and that the right-hand sides of
\eqref{eq:H12}--\eqref{eq:H23} vanish identically.  For each of the three
choices
\[
   (i,j,k)=(1,2,3),\quad (1,3,2),\quad (2,3,1),
\]
assume that the four numbers
\[
   \ell_i+\ell_k,\qquad \ell_i,\qquad
   \ell_j+\ell_k,\qquad \ell_j
\]
are pairwise distinct.
Then $U$ is monomial.
\end{lemma}

\begin{proof}
For the three choices of $(i,j,k)$ in the statement, the four displayed
numbers are, respectively, the frequencies of the four terms in
\eqref{eq:H12}, \eqref{eq:H13}, and \eqref{eq:H23}.  By the
linear-independence observation in Section~\ref{sec:zero-count},
\eqref{eq:H12} gives
\[
   u_{22}=\Delta_1=u_{11}=\Delta_2=0,
\]
and \eqref{eq:H13} gives
\[
   u_{33}=\Delta_1=u_{11}=\Delta_3=0.
\]
Similarly, \eqref{eq:H23} gives
\[
   u_{33}=\Delta_2=u_{22}=\Delta_3=0.
\]
Thus
\[
   u_{11}=u_{22}=u_{33}=0,
   \qquad
   \Delta_1=\Delta_2=\Delta_3=0.
\]
Because the diagonal entries vanish, the three principal minors reduce to
\[
   \Delta_1=-u_{23}u_{32},\qquad
   \Delta_2=-u_{13}u_{31},\qquad
   \Delta_3=-u_{12}u_{21}.
\]
Their vanishing therefore gives
\[
   u_{12}u_{21}=u_{13}u_{31}=u_{23}u_{32}=0.
\]

Fix $r\in\{1,2,3\}$ and suppose, for contradiction, that row $r$ contains
two nonzero entries.  Since $u_{rr}=0$, they must be $u_{rs}$ and $u_{rt}$,
where $\{r,s,t\}=\{1,2,3\}$.  The product relations above then imply
\[
   u_{sr}=0\qquad\text{and}\qquad u_{tr}=0.
\]
Together with $u_{rr}=0$, this says that every entry in column $r$ is zero,
contradicting the fact that a column of a unitary matrix has norm one.
Hence every row contains at most one nonzero entry.  On the other hand,
every row has norm one, so every row contains exactly one nonzero entry.

There are therefore exactly three nonzero entries in the whole matrix.
Every column of a unitary matrix also has norm one and hence contains at
least one nonzero entry.  Since the three nonzero entries are distributed
among three columns, each column contains exactly one.  Thus $U$ is
monomial.
\end{proof}

The preceding lemma settles the case in which none of the relevant
frequencies coincide.  The next lemma identifies the length relations that
can produce a frequency collision and hence require separate treatment.

\begin{lemma}\label{lem:frequency-collisions}
Let $i,j,k$ be distinct elements of $\{1,2,3\}$.  The four numbers
\[
   \ell_i+\ell_k,\qquad \ell_i,\qquad
   \ell_j+\ell_k,\qquad \ell_j,
\]
fail to be pairwise distinct if and only if at least one of the following
relations holds:
\[
   \ell_i=\ell_j,\qquad
   \ell_i=\ell_j+\ell_k,
   \qquad
   \ell_j=\ell_i+\ell_k.
\]
\end{lemma}

\begin{proof}
The equality
$\ell_i+\ell_k=\ell_j+\ell_k$ is equivalent to
$\ell_i=\ell_j$.  Because all three lengths are positive, neither
$\ell_i+\ell_k=\ell_i$ nor
$\ell_j+\ell_k=\ell_j$ is possible.  The only remaining equalities are
$\ell_i+\ell_k=\ell_j$ and
$\ell_j+\ell_k=\ell_i$, which are precisely the two displayed additive
relations.  The converse is immediate.
\end{proof}

We can now combine these two elementary observations with the global
cofactor identities from Proposition~\ref{prop:cofactor}.  The result is a
complete list of the boundary unitaries that must be considered in the
tiling argument.

\begin{proposition}\label{prop:classification}
After a simultaneous permutation of the interval labels and conjugation of
the boundary unitary $U$ by the corresponding permutation matrix, exactly
one of the following holds:
\begin{enumerate}[label=\textup{(\roman*)}]
\item $U$ is monomial;
\item $U$ is non-monomial and $\ell_1=\ell_2=\ell_3$;
\item $U$ is non-monomial, $\ell_3=\ell_1+\ell_2$, and
\begin{equation}\label{eq:special-U}
 U=\begin{pmatrix}
 0&\alpha&0\\
 c&0&d\\
 e&0&\alpha c
 \end{pmatrix},
 \qquad |\alpha|=1,
 \qquad
 \begin{pmatrix}c&d\\e&\alpha c\end{pmatrix}\in U(2).
\end{equation}
\end{enumerate}
\end{proposition}

\begin{proof}
Relabeling the intervals conjugates $U$ by the corresponding permutation
matrix and simultaneously permutes the lengths.

If $U$ is monomial, alternative \textup{(i)} holds, so assume from now on
that $U$ is non-monomial.  By the contrapositive of
Lemma~\ref{lem:monomial-from-minors}, at least one of its three four-tuples
fails to be pairwise distinct.  Lemma~\ref{lem:frequency-collisions} then
shows that either two interval lengths are equal or one interval length is
the sum of the other two.

If $\ell_1=\ell_2=\ell_3$, we obtain alternative \textup{(ii)}.  Suppose
next that the lengths are not all equal and that no one of them is the sum
of the other two.  The preceding reduction then forces exactly two lengths
to be equal.  Relabel so that
\[
   \ell_1=\ell_2=s,
   \qquad \ell_3=t\ne s.
\]
Equation \eqref{eq:H12} becomes, for every $\lambda\in\R$,
\[
 (u_{11}-u_{22})e^{2\pi i(s+t)\lambda}
 + (\Delta_1-\Delta_2)e^{2\pi is\lambda}=0.
\]
Since $t>0$, the frequencies $s+t$ and $s$ are distinct.  The
linear-independence observation in Section~\ref{sec:zero-count} therefore
gives
\[
   u_{11}=u_{22},
   \qquad \Delta_1=\Delta_2.
\]
Because the additive case is excluded,
\[
   \ell_3\ne\ell_1+\ell_2,
   \qquad\text{that is,}\qquad t\ne2s.
\]
Positivity, $s\ne t$, and $t\ne2s$ show that the four frequencies
\[
   2s,\quad s,\quad s+t,\quad t
\]
in \eqref{eq:H13} are distinct.  More explicitly, the only equalities not
immediately ruled out by positivity reduce to $2s=s+t$, $2s=t$, or $s=t$,
which would give $s=t$, $t=2s$, or $s=t$, respectively.  Hence
\[
   u_{33}=\Delta_1=u_{11}=\Delta_3=0,
\]
and then also $u_{22}=\Delta_2=0$.  These are the same vanishing conditions
used in the algebraic part of the proof of
Lemma~\ref{lem:monomial-from-minors}.  That argument shows that $U$ is
monomial, contrary to our standing assumption.  Hence the non-monomial case
cannot occur without an additive relation.

It remains to consider an additive relation.  Relabel so that
\[
   \ell_3=\ell_1+\ell_2.
\]
If $\ell_1\ne\ell_2$, the four frequencies in \eqref{eq:H12} are
\[
 2\ell_1+\ell_2,\qquad \ell_1,\qquad
 \ell_1+2\ell_2,\qquad \ell_2.
\]
They are pairwise distinct: equality of the first and third would give
$\ell_1=\ell_2$, while any other possible equality is excluded by
positivity or again reduces to $\ell_1=\ell_2$.  Linear independence in
\eqref{eq:H12} therefore gives
\[
   u_{11}=u_{22}=\Delta_1=\Delta_2=0.
\]
Since
$z_3(\lambda)=z_1(\lambda)z_2(\lambda)$ for every $\lambda\in\R$,
equation \eqref{eq:H13} then reduces, for every $\lambda\in\R$, to
\[
 -(u_{33}+\Delta_3)z_3(\lambda)=0,
\]
and $z_3(\lambda)$ never vanishes.  Hence $u_{33}+\Delta_3=0$.

If $\ell_1=\ell_2=s$, put
$z(\lambda)=e^{2\pi is\lambda}$ for $\lambda\in\R$.  Then
$z_1(\lambda)=z_2(\lambda)=z(\lambda)$ and
$z_3(\lambda)=z(\lambda)^2$.  For every $\lambda\in\R$, equations
\eqref{eq:H13} and \eqref{eq:H12} become, respectively,
\[
 \Delta_1z(\lambda)-(u_{33}+\Delta_3)z(\lambda)^2
 +u_{11}z(\lambda)^3=0
\]
and
\[
 (\Delta_1-\Delta_2)z(\lambda)
 +(u_{11}-u_{22})z(\lambda)^3=0.
\]
The functions $z(\lambda)$, $z(\lambda)^2$, and $z(\lambda)^3$ have the
distinct frequencies $s$, $2s$, and $3s$, because $s>0$.  Linear
independence therefore gives
\[
 \Delta_1=u_{11}=0,\qquad u_{33}+\Delta_3=0,
 \qquad \Delta_2=u_{22}=0.
\]
Thus, under $\ell_3=\ell_1+\ell_2$, both subcases
$\ell_1=\ell_2$ and $\ell_1\ne\ell_2$ give
\[
 u_{11}=u_{22}=\Delta_1=\Delta_2=0,
 \qquad u_{33}+\Delta_3=0.
\]
Since $u_{11}=u_{22}=0$, the definitions of the principal minors become
\[
 \Delta_1=-u_{23}u_{32},\qquad
 \Delta_2=-u_{13}u_{31},\qquad
 \Delta_3=-u_{12}u_{21}.
\]
Therefore these two subcases both satisfy
\begin{equation}\label{eq:additive-constraints}
 u_{11}=u_{22}=0,
 \qquad
 u_{23}u_{32}=0,
 \qquad
 u_{13}u_{31}=0,
 \qquad
 u_{33}=u_{12}u_{21}.
\end{equation}

Write
\[
 U=\begin{pmatrix}
 0&a&b\\
 c&0&d\\
 e&f&ac
 \end{pmatrix}.
\]
Comparing the norm of row $1$ with that of column $2$, and the norm of
row $2$ with that of column $1$, gives
\[
   |a|^2+|b|^2=|a|^2+|f|^2,
   \qquad
   |c|^2+|d|^2=|c|^2+|e|^2,
\]
and hence
\[
   |b|=|f|,
   \qquad |d|=|e|.
\]
The two product conditions in \eqref{eq:additive-constraints} are $be=0$
and $df=0$.  If $b=0$, then $f=0$.  If $b\ne0$, then $f\ne0$, and the
two product conditions force $e=d=0$.  Thus either $b=f=0$, or $d=e=0$.

Consider first the case $b=f=0$.  Then the first row of $U$ is
$(0,a,0)$.  Since every row of a unitary matrix has norm one, $|a|=1$.
After deleting row $1$ and column $2$, which contain the entry $a$, the
remaining submatrix is
\[
 \begin{pmatrix}c&d\\e&ac\end{pmatrix}\in U(2).
\]
The orthonormality of the remaining two rows and columns is therefore
exactly the unitarity of the displayed $2\times2$ submatrix.
This is \eqref{eq:special-U} with $\alpha=a$.  In the second case
$d=e=0$, so the norm of row $2$ gives $|c|=1$.  Interchanging labels $1$
and $2$ preserves the relation $\ell_3=\ell_1+\ell_2$ and conjugates $U$ to
\[
 \begin{pmatrix}
 0&c&0\\
 a&0&b\\
 f&0&ca
 \end{pmatrix},
\]
and the remaining block $\left(\begin{smallmatrix}a&b\\f&ca\end{smallmatrix}\right)$
is unitary by the same row-and-column argument.  This is again
\eqref{eq:special-U}, now with $\alpha=c$.

Consequently, whether $\ell_1=\ell_2$ or $\ell_1\ne\ell_2$, a possible
interchange of labels $1$ and $2$ puts $U$ in the form
\eqref{eq:special-U}.

The three alternatives are mutually exclusive: alternatives \textup{(ii)}
and \textup{(iii)} explicitly require $U$ to be non-monomial, and positive
equal lengths cannot satisfy
$\ell_3=\ell_1+\ell_2$.  The preceding cases also show that one of the
three alternatives always occurs.  This proves the classification.
\end{proof}

\section{A fiber lemma in multiplicities two and three}\label{sec:fiber}

In the two non-monomial cases of Proposition~\ref{prop:classification}, the
spectrum will be a union of, respectively, three or two cosets of a common
lattice, while the set will be a lattice multitile of the same
multiplicity.  The following lemma supplies the common final step: in these
two low multiplicities, the spectral structure of the fibers upgrades the
multitiling to an ordinary translational tiling.  The lemma is proved only for
$k=2$ and $k=3$; whether the same conclusion holds under these hypotheses for
$k>3$ is unknown.

For $h>0$ and a bounded measurable set $E\subset\R$, define the
\emph{$h\Z$-fiber} of $E$ by
\[
   A_x=\{n\in\Z:x+nh\in E\},
   \qquad 0\le x<h.
\]
\begin{lemma}\label{lem:fiber}
Let $k\in\{2,3\}$, let $h>0$, and let $E\subset\R$ be a bounded
measurable set.  Suppose that $E$ is a $k$-fold lattice multitile by $h\Z$:
\begin{equation}\label{eq:multitile}
   \sum_{n\in\Z}\1_E(x+nh)=k
   \qquad\text{for almost every }x.
\end{equation}
Suppose also that
\[
   \Lambda=\Gamma+h^{-1}\Z,
   \qquad |\Gamma|=k,
\]
is a spectrum of $E$, where the elements of $\Gamma$ are chosen to be distinct modulo $h^{-1}\Z$.  Then $E$ tiles $\R$ by translations.
\end{lemma}

\begin{proof}
For every $x\in[0,h)$, the sum in \eqref{eq:multitile} counts the integers
$n$ for which $x+nh\in E$, and hence equals $|A_x|$.  Thus there is a null
set $N\subset[0,h)$ such that $|A_x|=k$ for every
$x\in X:=[0,h)\setminus N$.  We first
show that, after passing to one common full-measure subset, all fiber
orthogonality relations hold simultaneously.  Fix
$\gamma,\gamma'\in\Gamma$ and put
\[
   \delta=\gamma-\gamma',
   \qquad
   S_{\gamma,\gamma'}(x)
   =\sum_{n\in A_x}e^{2\pi i h\delta n}.
\]
If $\gamma\ne\gamma'$, then $\gamma+q/h$ and $\gamma'$ are distinct
elements of $\Lambda$ for every $q\in\Z$, so the corresponding exponentials
are orthogonal in $L^2(E)$.  Hence
\[
 0
 =\int_E e^{2\pi i(\delta+q/h)t}\,dt.
\]
Every $t\in\R$ has a unique representation $t=x+nh$ with
$x\in[0,h)$ and $n\in\Z$.  Therefore, for every integrable function $g$,
decomposition into $h\Z$-fibers gives
\[
   \int_E g(t)\,dt
   =\int_0^h\sum_{n\in A_x}g(x+nh)\,dx.
\]
Applying this identity with $g(t)=e^{2\pi i(\delta+q/h)t}$ and using
$e^{2\pi iqn}=1$, we obtain
\[
 0=\int_0^h e^{2\pi iqx/h}e^{2\pi i\delta x}
       S_{\gamma,\gamma'}(x)\,dx
       \qquad(q\in\Z).
\]
Thus all Fourier coefficients of the $L^1([0,h))$ function
$e^{2\pi i\delta x}S_{\gamma,\gamma'}(x)$ vanish.  Fourier uniqueness in
$L^1([0,h))$ therefore gives
$e^{2\pi i\delta x}S_{\gamma,\gamma'}(x)=0$ almost everywhere.  Since
$e^{2\pi i\delta x}\ne0$ for every $x$, it follows that
$S_{\gamma,\gamma'}(x)=0$ almost everywhere.  When
$\gamma=\gamma'$, the multitiling hypothesis gives
$S_{\gamma,\gamma}(x)=|A_x|=k$ almost everywhere.

For each pair $(\gamma,\gamma')\in\Gamma^2$, the preceding identity holds
outside a null set that may depend on the pair.  Since $\Gamma^2$ is finite,
we may choose a single full-measure set $X_0\subset[0,h)$ such that, for every
$x\in X_0$,
\[
   |A_x|=k,
   \qquad
   S_{\gamma,\gamma'}(x)=
   \begin{cases}
      k,&\gamma=\gamma',\\
      0,&\gamma\ne\gamma',
   \end{cases}
   \qquad(\gamma,\gamma'\in\Gamma).
\]
Consequently, for every $x\in X_0$, the matrix
\begin{equation}\label{eq:fiber-matrix}
   \frac1{\sqrt{k}}
   \left(e^{2\pi i h\gamma n}\right)_{\gamma\in\Gamma,\ n\in A_x}
\end{equation}
is unitary after any ordering of the finite set $A_x$.

Suppose first that $k=2$.  Choose the two elements $\gamma_1,\gamma_2$ of $\Gamma$ and set
\[
   \alpha=h(\gamma_1-\gamma_2).
\]
For $x\in X_0$, write $A_x=\{m_x,n_x\}$.  Orthogonality of the
corresponding rows in \eqref{eq:fiber-matrix} gives
\[
   e^{2\pi i\alpha(m_x-n_x)}=-1.
\]
Fix $x_0\in X_0$.  Since $m_{x_0}-n_{x_0}\ne0$, the relation
$\alpha(m_{x_0}-n_{x_0})\in\frac12+\Z$ shows that $\alpha$ is rational.
Write $\alpha=p/Q$ in lowest terms, with $Q>0$.  The relation
$(p/Q)(m_{x_0}-n_{x_0})\in\frac12+\Z$ shows that $Q$ is even;
consequently $p$ is odd and has a multiplicative inverse $p^{-1}$ modulo
$Q$.  For every $x\in X_0$, the identity above is equivalent to
\[
   p(m_x-n_x)\equiv Q/2\pmod Q.
\]
Since $p^{-1}$ is odd, $p^{-1}Q/2\equiv Q/2\pmod Q$, and hence
\[
   m_x-n_x\equiv Q/2\pmod Q
   \qquad(x\in X_0).
\]
Thus $A_x$ modulo $Q$ is a coset of
\[
   H_2=\{0,Q/2\}\subset\Z/Q\Z.
\]
Its two elements have distinct residues modulo $Q$, since their difference
is congruent to $Q/2$.  Hence each residue in this coset is represented
exactly once by $A_x$.
The set
\[
   B_2=\{0,1,\ldots,Q/2-1\}+Q\Z
\]
is a common tiling complement for all the fibers:
$A_x\oplus B_2=\Z$ for every $x\in X_0$.  Here $A\oplus B=\Z$ means that
every integer has a unique representation $a+b$ with $a\in A$ and $b\in B$.

Now suppose that $k=3$.  Choose two distinct elements $\gamma_1,\gamma_2\in\Gamma$ and again put $\alpha=h(\gamma_1-\gamma_2)$.  Orthogonality of the corresponding rows in \eqref{eq:fiber-matrix} gives
\[
   \sum_{n\in A_x}e^{2\pi i\alpha n}=0.
\]
Three unit complex numbers whose sum is zero form a rotated regular triangle.
Indeed, after a common rotation write them as $1,u,v$.  Then $u+v=-1$
and $|u|=|v|=1$, so $|u+v|^2=1$ gives
$\operatorname{Re}(u\overline v)=-1/2$; their arguments are therefore
separated by $2\pi/3$.
Therefore, for every $x\in X_0$,
\[
   \{e^{2\pi i\alpha n}:n\in A_x\}
\]
is a coset of $\{1,\omega,\omega^2\}$, where $\omega=e^{2\pi i/3}$.
Fix $x_0\in X_0$.  Among the three elements of $A_{x_0}$, choose distinct
$n_1,n_2$ such that
\[
   e^{2\pi i\alpha(n_1-n_2)}\in\{\omega,\omega^2\}.
\]
Setting $d=n_1-n_2\ne0$, we obtain
\[
 \alpha d\in\Z\mathbin{\pm}\tfrac13.
\]
Writing $\alpha=p/Q$ in lowest terms, the displayed relation implies
$3\mid Q$.  Now fix $x\in X_0$ and $n_0\in A_x$.  Dividing the three
phases associated with $A_x$ by $e^{2\pi i\alpha n_0}$ gives
\[
   \{e^{2\pi i\alpha(n-n_0)}:n\in A_x\}
   =\{1,\omega,\omega^2\}.
\]
Since $\alpha=p/Q$, this is equivalent to
\[
 p(A_x-n_0)\equiv\{0,Q/3,2Q/3\}\pmod Q.
\]
Multiplication by $p$ is an automorphism of $\Z/Q\Z$ and preserves its unique
subgroup of order three, namely $\{0,Q/3,2Q/3\}$.  Multiplying the preceding
congruence by $p^{-1}$ therefore gives
\[
   A_x-n_0\equiv\{0,Q/3,2Q/3\}\pmod Q.
\]
Thus $A_x$ modulo $Q$ is the coset
$n_0+\{0,Q/3,2Q/3\}$ in $\Z/Q\Z$.
The three residues are distinct and $|A_x|=3$, so each is represented
exactly once by an element of $A_x$.
The common complement is
\[
   B_3=\{0,1,\ldots,Q/3-1\}+Q\Z.
\]
Thus $A_x\oplus B_3=\Z$ for every $x\in X_0$.

Let us summarize the two cases.  For $k\in\{2,3\}$, the set $B=B_k$
constructed above is a common tiling complement for the fibers, so that
$A_x\oplus B=\Z$ for every $x\in X_0$.  Let $t\in\R$ have the unique
representation
\[
   t=x+hn,
   \qquad x\in[0,h),\quad n\in\Z.
\]
If $x\in X_0$, then, for each $b\in B$,
\[
   t-hb=x+h(n-b)\in E
   \quad\Longleftrightarrow\quad
   n-b\in A_x.
\]
Consequently,
\[
 \sum_{b\in B}\1_E(t-hb)
 =\#\{(a,b)\in A_x\times B:a+b=n\}=1.
\]
Since $X_0$ has full measure in $[0,h)$, it follows that
\[
   \sum_{b\in B}\1_E(t-hb)=1
   \qquad\text{for almost every }t\in\R,
\]
so $E$ tiles $\R$ by the translation set $hB$.
\end{proof}

\section{The two-component case as a warm-up}\label{sec:fewer}

Before completing the three-component case, we illustrate the preceding
method for two components.  The resulting secular matrix has size two, so
the same ideas require much less case analysis.  The one-component case is
immediate and is recorded briefly below.

The equivalence between spectrality and translational tiling for a union of
two intervals was proved by \L aba \cite{laba-two}.  A short proof of the
spectral-to-tiling direction was subsequently given by Bose and Madan
\cite{bose-madan}, and Kolountzakis discusses the result and a more recent
weak-tiling approach in \cite[Section~3.7]{kolountzakis-survey}.  We include a
different proof of the spectral-to-tiling direction.  It uses the boundary
unitary, the secular determinant, the density uniqueness principle, and
Lemma~\ref{lem:fiber}, and thus serves as a warm-up for the three-component
argument.

Throughout this section, $\Omega$ is assumed to have spectrum $\Lambda$.

If $\Omega=[a,b)$ is one interval of length $L=b-a$, the one-dimensional
version of the boundary-operator construction gives a scalar unitary
$\eta\in\C$, with $|\eta|=1$, and the boundary condition
$f(b)=\eta f(a)$.  Applying this condition to the exponential eigenfunction
$e_\lambda$ gives
\[
   e^{2\pi i\lambda b}=\eta e^{2\pi i\lambda a},
\]
or, equivalently,
\[
   e^{2\pi i\lambda L}=\eta.
\]
Choose $\theta\in\R$ such that $\eta=e^{2\pi i\theta}$.  The solutions of
the displayed equation are exactly
\[
   \lambda=\frac{\theta+n}{L},
   \qquad n\in\Z.
\]
Since the zeros of the secular equation are precisely the spectral points,
\[
   \Lambda=\frac{\theta}{L}+L^{-1}\Z.
\]
Thus the spectrum is one coset of the lattice $L^{-1}\Z$.  The interval
tiles by $L\Z$.

Suppose that $\Omega=[a_1,b_1)\cup[a_2,b_2)$ has two components, where
$\ell_j=b_j-a_j>0$ for $j=1,2$, and put $L=\ell_1+\ell_2$.  For
$\zeta\in\C$, set
\[
   z_j(\zeta)=e^{2\pi i\zeta\ell_j},
   \qquad j=1,2.
\]
The same operator construction gives a boundary unitary
$U=(u_{jk})\in U(2)$ and
\[
   F(\zeta)=\det\bigl(\diag(z_1(\zeta),z_2(\zeta))-U\bigr),
   \qquad \zeta\in\C.
\]
The arguments used to prove Lemma~\ref{lem:boundary-unitary} and
Proposition~\ref{prop:secular}\textup{(i), (ii), and (iv)} do not depend on
there being exactly three components.  Applied in dimension two, they show
that the zeros of $F$ are precisely $\Lambda$, are all real and simple, and
satisfy
\[
 \#(\Lambda\cap[-R,R])=2LR+O(1).
\]
Indeed, expanding the $2\times2$ determinant gives
\[
\begin{aligned}
 F(\zeta)
 &=z_1(\zeta)z_2(\zeta)-u_{22}z_1(\zeta)
   -u_{11}z_2(\zeta)+\det U\\
 &=e^{2\pi iL\zeta}-u_{22}e^{2\pi i\ell_1\zeta}
   -u_{11}e^{2\pi i\ell_2\zeta}+\det U.
\end{aligned}
\]
Thus the term of frequency $0$ has coefficient $\det U\ne0$, because $U$
is unitary, while the term of largest frequency
$L=\ell_1+\ell_2$ has coefficient $1$.  Hence the frequency width of $F$ is
$L$.  The same left- and
right-kernel calculation as in Proposition~\ref{prop:cofactor} gives, at
every spectral point, the identity between the two diagonal cofactors
\[
   z_1(\lambda)\bigl(z_2(\lambda)-u_{22}\bigr)
   =z_2(\lambda)\bigl(z_1(\lambda)-u_{11}\bigr),
   \qquad \lambda\in\Lambda.
\]
Thus
\[
   u_{11}z_2(\lambda)-u_{22}z_1(\lambda)=0,
   \qquad \lambda\in\Lambda.
\]
To extend this equality from the spectral points to every parameter, define
the exponential polynomial
\[
   H(\zeta)=u_{11}e^{2\pi i\ell_2\zeta}
            -u_{22}e^{2\pi i\ell_1\zeta},
   \qquad \zeta\in\C.
\]
It vanishes on $\Lambda$, and its only possible frequencies are $\ell_1$
and $\ell_2$; hence its frequency width is at most
$|\ell_1-\ell_2|<L$.  Since
$\#(\Lambda\cap[-R,R])=2LR+O(1)$, Corollary~\ref{cor:density-unique}
implies that $H$ vanishes identically.  This is the same uniqueness argument
used in the proof of Proposition~\ref{prop:cofactor}, now applied to the
$2\times2$ cofactor identity.  Consequently,
\[
   u_{11}z_2(\zeta)-u_{22}z_1(\zeta)=0
   \qquad\text{for every }\zeta\in\C.
\]

The two possible relations between the component lengths now lead to the
same two tiling mechanisms that will occur below: unequal lengths force a
monomial boundary unitary, whereas equal lengths produce a two-fold lattice
multitiling.

If $\ell_1\ne\ell_2$, linear independence of the two exponentials gives
$u_{11}=u_{22}=0$.  Unitarity then gives
\[
 U=\begin{pmatrix}0&\eta_1\\ \eta_2&0\end{pmatrix},
 \qquad |\eta_1|=|\eta_2|=1,
\]
and hence
\[
 F(\lambda)=z_1(\lambda)z_2(\lambda)-\eta_1\eta_2
 =e^{2\pi iL\lambda}-\eta_1\eta_2.
\]
Since $|\eta_1\eta_2|=1$, choose $\lambda_0\in\R$ such that
$e^{2\pi iL\lambda_0}=\eta_1\eta_2$.  The zero-set characterization of the
spectrum then gives
$\Lambda=\lambda_0+L^{-1}\Z$.  Comparing the two boundary relations
requires no implicit change of convention: for every $\lambda\in\Lambda$,
the condition $f(\mathbf b)=Uf(\mathbf a)$ applied to $e_\lambda$ reads
\[
 e^{2\pi i\lambda b_1}=\eta_1e^{2\pi i\lambda a_2},
 \qquad
 e^{2\pi i\lambda b_2}=\eta_2e^{2\pi i\lambda a_1}.
\]
Since $\lambda+L^{-1}\in\Lambda$, dividing these identities at
$\lambda+L^{-1}$ by the corresponding identities at $\lambda$ gives
\[
 b_1-a_2\in L\Z,
 \qquad b_2-a_1\in L\Z.
\]
The left-endpoint multiset $\{a_1,a_2\}$ and the right-endpoint multiset
$\{b_1,b_2\}$ therefore agree modulo $L\Z$.  Hence the $L$-periodization
\[
   N(x)=\sum_{n\in\Z}\1_\Omega(x+nL),
   \qquad x\in\R,
\]
is constant almost everywhere on $\R/L\Z$: its distributional derivative is
the sum of the Dirac masses at the left endpoints minus those at the right
endpoints, which cancel modulo $L\Z$.  Its constant value is
\[
   \frac1L\int_0^L N(x)\,dx=\frac{|\Omega|}{L}=1.
\]
Thus $\Omega$ tiles by the translation set $L\Z$.

If $\ell_1=\ell_2=h$, set $z(\zeta)=e^{2\pi ih\zeta}$ for $\zeta\in\C$.
Then $F(\zeta)=\det(z(\zeta)I-U)$.  Denote the two eigenvalues of $U$ by
$\rho_1$ and $\rho_2$.  They lie on the unit circle and are distinct, since
a repeated eigenvalue would give multiple zeros of $F$.  For each
$j\in\{1,2\}$, choose $\gamma_j\in\R$ such that
\[
   e^{2\pi ih\gamma_j}=\rho_j.
\]
The solutions of $z(\lambda)=\rho_j$ are then precisely
$\gamma_j+h^{-1}\Z$.  Since the zeros of $F$ are exactly the spectral
points, it follows that
\[
   \Lambda=\Gamma+h^{-1}\Z,
   \qquad |\Gamma|=2.
\]
Here $\Gamma=\{\gamma_1,\gamma_2\}$, and its two elements are distinct modulo
$h^{-1}\Z$ because $\rho_1\ne\rho_2$.
Each of the two intervals has length $h$ and therefore periodizes to one
under $h\Z$.  Consequently, $\Omega$ is a two-fold $h\Z$-multitile.
Lemma~\ref{lem:fiber} with $k=2$ now gives a translational tiling.

We now turn to the three-component case, where Proposition~\ref{prop:classification}
provides the additional algebraic input.

\section{Completion for three components}\label{sec:completion}

We now return to $\Omega$ and combine Proposition~\ref{prop:classification}
with the fiber result in Lemma~\ref{lem:fiber}.  The three alternatives in the classification are
handled separately: monomial boundary unitaries give a lattice tiling
directly, while the two non-monomial alternatives give multitilings to which
Lemma~\ref{lem:fiber} applies.

\subsection{Monomial boundary unitaries}

Monomial boundary matrices are closely related to weighted permutation
actions and groups of local translations; see
\cite{dutkay-jorgensen-momentum,ducasse-dutkay-fernandez}.  We first show
that a monomial boundary unitary forces $\Omega$, up to a null set, to be a
fundamental domain for the lattice $L\Z$.

\begin{proposition}\label{prop:monomial-tiles}
If the boundary unitary $U$ is monomial, then $\Omega$ tiles by the lattice
$L\Z$.  Equivalently, up to a null set, $\Omega$ is a measurable fundamental
domain for $L\Z$ in $\R$.
\end{proposition}

\begin{proof}
Since $U$ is monomial, for each $i\in\{1,2,3\}$ there is a unique column
$\pi(i)\in\{1,2,3\}$ in which the $i$th row of $U$ is nonzero.  Each column
also contains exactly one nonzero entry, so the map
$\pi:\{1,2,3\}\to\{1,2,3\}$ is a permutation.  Thus
\[
   U_{i,\pi(i)}=\eta_i,
   \qquad |\eta_i|=1,
   \qquad U_{ij}=0\quad\text{if }j\ne\pi(i).
\]
For $\lambda\in\R$, write
$y=(y_1,y_2,y_3)^{\mathsf T}$ for a possible left-endpoint vector.  By
\eqref{eq:F}, the eigenvalue equation $M(\lambda)y=0$, equivalently
$D_\ell(\lambda)y=Uy$, has the componentwise form
\[
 z_i(\lambda)y_i=\eta_i y_{\pi(i)}.
\]
Let $C=(i_1\,i_2\,\ldots\,i_r)$ be one of the cycles in the disjoint-cycle
decomposition of $\pi$; thus $\pi(i_s)=i_{s+1}$ for $1\le s<r$ and
$\pi(i_r)=i_1$.  Put
$L_C=\sum_{i\in C}\ell_i$ and $\eta_C=\prod_{i\in C}\eta_i$.
Because $\pi$ maps $C$ onto itself, the equations involving the coordinates
indexed by $C$ are
\[
   z_{i_s}(\lambda)y_{i_s}=\eta_{i_s}y_{i_{s+1}}
   \quad(1\le s<r),
   \qquad
   z_{i_r}(\lambda)y_{i_r}=\eta_{i_r}y_{i_1},
\]
and they do not involve any coordinate outside $C$.  Starting with a nonzero
value of $y_{i_1}$, the first $r-1$ equations determine
$y_{i_2},\ldots,y_{i_r}$ successively.  The last equation is then satisfied
if and only if
\[
   e^{2\pi i\lambda L_C}=\eta_C,
\]
where $|\eta_C|=1$.  This equation has infinitely many real solutions.  For
each such $\lambda$, choose a nonzero solution $y$ of the preceding cycle
equations and set $y_i=0$ for $i\notin C$.  Define
\[
   f(x)=y_i e^{2\pi i\lambda(x-a_i)}
   \qquad (x\in[a_i,b_i],\ 1\le i\le3).
\]
Then $f(\mathbf b)=Uf(\mathbf a)$, so $f\in\Dom(A)$ by
Lemma~\ref{lem:boundary-unitary}, and $Af=Df=\lambda f$.  Moreover, $f$
vanishes on every interval whose index is not in $C$.  If $C$ were a proper
cycle, this would give a nonzero eigenfunction vanishing on at least one
component.  This is impossible by the definition of $A$.  Indeed, if
$Ag=\lambda g$ and
\[
   g=\sum_{\nu\in\Lambda}c_\nu e_\nu,
\]
then comparison of the spectral coefficients gives
$(\nu-\lambda)c_\nu=0$ for every $\nu\in\Lambda$.  Hence
$\lambda\in\Lambda$ and $g\in\operatorname{span}\{e_\lambda\}$.  But
$|e_\lambda(x)|=1$ on every component of $\Omega$, so no nonzero multiple of
$e_\lambda$ can vanish on a component.  Thus $\pi$ is a single three-cycle.

Multiplying the three boundary equations around the cycle gives
\begin{equation}\label{eq:monomial-cycle-spectrum}
   e^{2\pi i\lambda L}=\eta_1\eta_2\eta_3.
\end{equation}
Conversely, suppose that $\lambda\in\R$ satisfies
\eqref{eq:monomial-cycle-spectrum}.  Choose
$y_{i_1}\ne0$ and define successively
\[
   y_{i_{s+1}}
   =\frac{z_{i_s}(\lambda)}{\eta_{i_s}}y_{i_s}
   \qquad(1\le s<3).
\]
Equation~\eqref{eq:monomial-cycle-spectrum} is exactly the condition that the
remaining boundary equation
$z_{i_3}(\lambda)y_{i_3}=\eta_{i_3}y_{i_1}$ also holds.  Hence
$D_\ell(\lambda)y=Uy$, and the function
$f(x)=y_i e^{2\pi i\lambda(x-a_i)}$ belongs to $\Dom(A)$ and satisfies
$Af=\lambda f$.  Thus every real solution of
\eqref{eq:monomial-cycle-spectrum} is an eigenvalue of $A$, and therefore
\[
   \Lambda=\lambda_0+L^{-1}\Z
\]
for some $\lambda_0$.  For each $\lambda\in\Lambda$, apply the boundary
condition $e_\lambda(\mathbf b)=Ue_\lambda(\mathbf a)$ to the eigenfunction
$e_\lambda$.  Its $i$th component is
\begin{equation}\label{eq:monomial-component-boundary}
   e^{2\pi i\lambda b_i}
   =\eta_i e^{2\pi i\lambda a_{\pi(i)}}.
\end{equation}
The identity $\Lambda=\lambda_0+L^{-1}\Z$ implies that
$\lambda+L^{-1}\in\Lambda$ whenever $\lambda\in\Lambda$.  Apply
\eqref{eq:monomial-component-boundary} at $\lambda+L^{-1}$ to obtain
\[
   e^{2\pi i(\lambda+L^{-1})b_i}
   =\eta_i e^{2\pi i(\lambda+L^{-1})a_{\pi(i)}}.
\]
Dividing this equality by \eqref{eq:monomial-component-boundary} gives
\[
   e^{2\pi i b_i/L}=e^{2\pi i a_{\pi(i)}/L},
\]
and hence
\[
   b_i-a_{\pi(i)}\in L\Z
   \qquad (i=1,2,3).
\]
Thus $b_i\equiv a_{\pi(i)}\pmod L$ for each $i$.  Since $\pi$ is a
permutation, this means that the endpoint multisets satisfy
\[
   \{b_1\bmod L,b_2\bmod L,b_3\bmod L\}
   =\{a_1\bmod L,a_2\bmod L,a_3\bmod L\}.
\]

Consider the $L$-periodization
\[
   N(x)=\sum_{n\in\Z}\1_\Omega(x+nL).
\]
On the circle $\R/L\Z$, its distributional derivative is
\[
 N'=\sum_{j=1}^3
 \bigl(\delta_{a_j\bmod L}-\delta_{b_j\bmod L}\bigr)=0.
\]
A distribution on the circle with zero derivative is constant; since
$N\in L^1(\R/L\Z)$, the constant is represented almost everywhere.  Its
average is
\[
   \frac1L\int_0^L N(x)\,dx=\frac{|\Omega|}{L}=1.
\]
Therefore $N=1$ almost everywhere, and $\Omega$ tiles by the translation set
$L\Z$.
\end{proof}

\subsection{Three equal lengths}

We next treat the first non-monomial alternative in
Proposition~\ref{prop:classification}.  Equal lengths make the secular
matrix a scalar multiple of the identity minus $U$, so its zeros are
determined directly by the eigenvalues of $U$.

Assume
\[
   \ell_1=\ell_2=\ell_3=h.
\]
Then $D_\ell(\lambda)=z(\lambda)I$, where $z(\lambda)=e^{2\pi ih\lambda}$, and
\[
   F(\lambda)=\det(z(\lambda)I-U).
\]
The eigenvalues of the boundary unitary $U$ lie on the unit circle.  If an
eigenvalue $\rho$ of $U$ had algebraic multiplicity greater than one, every
solution of $z(\lambda)=\rho$ would be a multiple zero of $F$, since
$z'(\lambda)=2\pi ihz(\lambda)\ne0$.
Proposition~\ref{prop:secular}\textup{(ii)} excludes this.  Thus the three
eigenvalues of $U$ are distinct; denote them by $\rho_1,\rho_2,\rho_3$.
For each $j\in\{1,2,3\}$, choose the unique
$\gamma_j\in[0,h^{-1})$ such that
\[
   e^{2\pi ih\gamma_j}=\rho_j.
\]
Since the real solutions of $e^{2\pi ih\lambda}=\rho_j$ are precisely
$\gamma_j+h^{-1}\Z$, we obtain
\[
   \Lambda=\Gamma+h^{-1}\Z,
   \qquad
   \Gamma=\{\gamma_1,\gamma_2,\gamma_3\},
\]
where the elements of $\Gamma$ are distinct modulo $h^{-1}\Z$.
Each interval of length $h$ periodizes to one under $h\Z$, so
\[
   \sum_{n\in\Z}\1_\Omega(x+nh)=3
\]
for almost every $x\in\R$.  Applying Lemma~\ref{lem:fiber} with
$E=\Omega$, $k=3$, and $\Gamma=\{\gamma_1,\gamma_2,\gamma_3\}$, we obtain
a set $B\subset\Z$ such that
\[
   \sum_{b\in B}\1_\Omega(x-hb)=1
   \qquad\text{for almost every }x\in\R.
\]
Thus $\Omega$ tiles $\R$ by the translation set $hB$.

\subsection{The additive exceptional case}

It remains to treat the second non-monomial alternative.  Here the additive
relation among the lengths reduces the secular determinant to a quadratic
polynomial in a single exponential variable, so that the spectrum becomes
the union of two cosets of one lattice.

Assume
\[
   \ell_3=\ell_1+\ell_2=:h
\]
and that $U$ is not monomial.  By Proposition~\ref{prop:classification}, after possibly interchanging the first two labels,
\[
 U=\begin{pmatrix}
 0&\alpha&0\\
 c&0&d\\
 e&0&\alpha c
 \end{pmatrix},
 \qquad |\alpha|=1,
 \qquad
 \begin{pmatrix}c&d\\e&\alpha c\end{pmatrix}\in U(2).
\]
Since $z_3=z_1z_2$, direct expansion gives
\begin{equation}\label{eq:additive-F}
   F(\lambda)
   =\bigl(z_3(\lambda)-\alpha c\bigr)^2-\alpha de.
\end{equation}
Put
\[
 q(w)=(w-\alpha c)^2-\alpha de.
\]
Its constant term is $q(0)=-\det U\ne0$, so both roots are nonzero.  If
$q(\rho)=0$, choose $\lambda\in\C$ such that
\[
   e^{2\pi ih\lambda}=\rho,
\]
which is possible because $\rho\ne0$.  Then
$F(\lambda)=q(e^{2\pi ih\lambda})=q(\rho)=0$, so
$\lambda\in\R$ by Proposition~\ref{prop:secular}\textup{(ii)}.  Therefore
$|\rho|=|e^{2\pi ih\lambda}|=1$.  The two roots of $q$ must be distinct.
Indeed, if $q$ had a
repeated root $\rho$, then $q(w)=(w-\rho)^2$.  Choose
$\lambda_0\in\R$ such that $e^{2\pi ih\lambda_0}=\rho$.  Since
\[
   \left.\frac{d}{d\lambda}e^{2\pi ih\lambda}\right|_{\lambda=\lambda_0}
   =2\pi ih\rho\ne0,
\]
the identity $F(\lambda)=q(e^{2\pi ih\lambda})$ would make $\lambda_0$ a
zero of $F$ of multiplicity two, contradicting the simplicity of the zeros
of $F$ in Proposition~\ref{prop:secular}\textup{(ii)}.  Denote the two
distinct roots of $q$ by $\rho_1$ and $\rho_2$.  For each $j\in\{1,2\}$,
choose the unique $\gamma_j\in[0,h^{-1})$ such that
\[
   e^{2\pi ih\gamma_j}=\rho_j.
\]
The real solutions of $e^{2\pi ih\lambda}=\rho_j$ are precisely
$\gamma_j+h^{-1}\Z$.  Hence
\[
   \Lambda=(\gamma_1+h^{-1}\Z)\,\sqcup\,
           (\gamma_2+h^{-1}\Z).
\]

The first row of the boundary equation \eqref{eq:global-boundary} is
\[
   e^{2\pi i\lambda b_1}
   =\alpha e^{2\pi i\lambda a_2}.
\]
Since $\Lambda$ is the disjoint union of two cosets of $h^{-1}\Z$, it is
invariant under addition by $h^{-1}$.  Thus, for every $\lambda\in\Lambda$,
the first-row boundary equation at $\lambda+h^{-1}$ is
\[
   e^{2\pi i(\lambda+h^{-1})b_1}
   =\alpha e^{2\pi i(\lambda+h^{-1})a_2}.
\]
Dividing this equality by the first-row equation at $\lambda$ gives
\[
   e^{2\pi i b_1/h}=e^{2\pi i a_2/h},
\]
and hence
\[
   b_1-a_2\in h\Z.
\]
Modulo $h\Z$,
\[
 a_2\equiv b_1\equiv a_1+\ell_1,
\]
so the left endpoint $a_2$ of the second interval coincides modulo $h\Z$
with the right endpoint $b_1$ of the first interval.  Moreover,
\[
   b_2\equiv a_2+\ell_2
        \equiv a_1+\ell_1+\ell_2
        \equiv a_1\pmod h.
\]
Thus $[a_1,b_1)$ and $[a_2,b_2)$ project to adjacent half-open arcs of total
length $h$ on $\R/h\Z$, and their indicator functions sum to one almost
everywhere on the circle; equivalently,
\[
   \sum_{n\in\Z}
   \left(
      \1_{[a_1,b_1)}(x+nh)+\1_{[a_2,b_2)}(x+nh)
   \right)=1
   \qquad\text{for almost every }x\in\R.
\]
Since $b_3-a_3=h$, the third interval also satisfies
\[
   \sum_{n\in\Z}\1_{[a_3,b_3)}(x+nh)=1
   \qquad\text{for almost every }x\in\R.
\]
Adding the last two periodization identities gives
\[
   \sum_{n\in\Z}\1_\Omega(x+nh)=2
\]
for almost every $x\in\R$.  Applying Lemma~\ref{lem:fiber} with
$E=\Omega$, $k=2$, and
$\Gamma=\{\gamma_1,\gamma_2\}$, we obtain a set $B\subset\Z$ such that
\[
   \sum_{b\in B}\1_\Omega(x-hb)=1
   \qquad\text{for almost every }x\in\R.
\]
Thus $\Omega$ tiles $\R$ by the translation set $hB$.

Proposition~\ref{prop:classification} shows that these three mechanisms
exhaust all possibilities.  Together with the lower-component cases recorded
in Section~\ref{sec:fewer}, this completes the proof of
Theorem~\ref{thm:main}; the three-component conclusions also give the
geometric alternatives in Theorem~\ref{thm:geometric-trichotomy}.

\section{Extensions and limitations for \texorpdfstring{$n$}{n} intervals}
\label{sec:n-intervals}

We conclude by separating the parts of the argument that apply to an
arbitrary finite union of intervals from those that rely essentially on the
number of components being at most three.  Let
\[
   E=\bigcup_{j=1}^n[a_j,b_j),
   \qquad \ell_j=b_j-a_j>0,
   \qquad L=\sum_{j=1}^n\ell_j,
\]
and suppose that $E$ has a spectrum $\Lambda$.

\paragraph{The boundary-operator framework.}
The construction in Section~\ref{sec:operator} is unchanged.  The derivative
defined diagonally in the exponential basis is a self-adjoint extension of
the minimal derivative, and there is a unique matrix $U\in U(n)$ such that
\[
   \Dom(A)=\left\{f\in\bigoplus_{j=1}^nH^1([a_j,b_j]):
   f(\mathbf b)=Uf(\mathbf a)\right\}.
\]
With
\[
   D_\ell(\lambda)
   =\diag\bigl(e^{2\pi i\lambda\ell_1},\ldots,
                    e^{2\pi i\lambda\ell_n}\bigr),
\]
the secular determinant
\[
   F(\lambda)=\det(D_\ell(\lambda)-U)
\]
still satisfies
\[
   F(\lambda)=0\quad\Longleftrightarrow\quad\lambda\in\Lambda
   \qquad(\lambda\in\C).
\]
Its zeros are real and simple, and
$\operatorname{rank}(D_\ell(\lambda)-U)=n-1$ at every
$\lambda\in\Lambda$.  Indeed, self-adjointness and the one-dimensional
eigenspaces of the diagonal operator $A$ give reality and rank $n-1$.  If
$y=(e^{2\pi i\lambda a_j})_{j=1}^n$, the transversality computation in
\eqref{eq:transversal} becomes
\[
   (D_\ell(\lambda)y)^*M'(\lambda)y
   =2\pi i\sum_{j=1}^n\ell_j=2\pi iL\ne0,
\]
and Jacobi's cofactor formula gives $F'(\lambda)\ne0$.  Thus the simplicity
argument also extends without change in substance.

\paragraph{Zero counting and cofactor identities.}
Expanding $F$ gives an exponential polynomial whose frequencies are subset
sums of $\ell_1,\ldots,\ell_n$.  Its extreme frequencies are $0$ and $L$.
Consequently, the zero-counting argument and the frequency-width uniqueness
principle of Section~\ref{sec:zero-count} remain available for every finite
$n$.  At a spectral point, the secular matrix has rank $n-1$, so its adjugate
has rank one.  Thus the adjugate and cofactor method of
Section~\ref{sec:cofactor} also has an $n$-dimensional formulation.  What does
not extend automatically is the subsequent classification: for larger $n$,
the resulting identities involve many principal minors and many possible
relations among subset sums of the component lengths.

\paragraph{Monomial boundary unitaries.}
The proof of Proposition~\ref{prop:monomial-tiles} extends to every finite
$n$.  If $U\in U(n)$ is
monomial, its underlying permutation cannot contain a proper cycle, since
such a cycle would produce a nonzero eigenfunction supported on only the
corresponding components.  Hence the permutation is a single $n$-cycle.
Multiplication of the boundary equations around that cycle gives
\[
   \Lambda=\lambda_0+L^{-1}\Z.
\]
Comparing the boundary equations at $\lambda$ and $\lambda+L^{-1}$ shows
that the multisets of left and right endpoints agree modulo $L\Z$.
Periodization then gives
\[
   \sum_{m\in\Z}\1_E(x+mL)=1
   \qquad\text{for almost every }x\in\R.
\]
Thus, for every finite $n$, a monomial boundary unitary forces $E$ to be a
fundamental domain for $L\Z$.

\paragraph{Equal component lengths.}
If $\ell_1=\cdots=\ell_n=h$, then
\[
   F(\lambda)=\det(e^{2\pi ih\lambda}I-U).
\]
The simplicity of the zeros of $F$ implies that the $n$ eigenvalues of $U$
are distinct.  Therefore $\Lambda$ is a union of $n$ distinct cosets of
$h^{-1}\Z$, while the geometry gives
\[
   \sum_{m\in\Z}\1_E(x+mh)=n
   \qquad\text{for almost every }x\in\R.
\]
Hence $E$ is an $n$-fold lattice multitile by $h\Z$.  For $n=2,3$,
Lemma~\ref{lem:fiber} converts this information into an ordinary
translational tiling.  No such conversion is established here for
$n\ge4$.

\paragraph{Potential extensions to $n$ intervals.}
The exhaustive classification in Proposition~\ref{prop:classification} is
specific to $U(3)$ and to the subset-sum relations among three positive
lengths.  Likewise, the final fiber argument uses the rigidity of zero sums
of two or three unit complex numbers: two such numbers must be antipodal,
and three must form a rotated equilateral triangle.  Starting with four unit
complex numbers, zero-sum configurations have continuous degrees of freedom,
so the proof of Lemma~\ref{lem:fiber} does not extend by the same argument.
In particular, the present method supplies a general operator, determinant,
zero-counting, and rank-one-cofactor framework for finite unions of
intervals, but the algebraic classification and the passage from higher
multiplicity lattice multitilings to ordinary tilings remain the obstacles
for $n\ge4$.

\section*{Acknowledgments}

The author is grateful to Chun-Kit Lai and Mihalis Kolountzakis for many
helpful discussions.  Part of this work was carried out during a visit to the
University of Crete.  The author thanks Mihalis Kolountzakis and Maria
Loukaki for their warm hospitality and the University of Crete for providing
an excellent working environment.  This work was supported by the National
Natural Science Foundation of China (NSFC), grant Nos.~12571198 and 12231013.

\bigskip
\noindent
Ruxi Shi\\
Shanghai Center for Mathematical Sciences, Fudan University,\\
200438 Shanghai, China\\
Email: \href{mailto:ruxishi@fudan.edu.cn}{\texttt{ruxishi@fudan.edu.cn}}

\end{document}